\documentclass[11pt,leqno]{article} 

\usepackage{palatino}
\usepackage{authblk}
\usepackage{braket}
\usepackage{amsfonts}
\usepackage{amssymb}
\usepackage{amsmath} 
\usepackage{latexsym}
\usepackage{amsthm}
\usepackage[usenames]{color}
\usepackage{hyperref}
\usepackage{tabularx}

\usepackage{mathtools}

\newcommand{\dsum}{\displaystyle\sum}

\hypersetup{pdfpagemode=UseNone}
 
\newcommand{\comment}[1]{{}}

\newcommand{\calG}{\mathcal{G}}
\newcommand{\calI}{\mathcal{I}}

\newcommand{\calK}{\mathcal{K}}
\newcommand{\calS}{\mathcal{S}}
\newcommand{\calP}{\mathcal{P}}
\newcommand{\calQ}{\mathcal{Q}}
\newcommand{\calL}{\mathcal{L}}
\newcommand{\calC}{\mathcal{C}}

\newcommand{\tX}{\tilde{X}}
\newcommand{\tY}{\tilde{Y}}

\newcommand{\tr}{\mathrm{Tr}}
\newcommand{\vecc}{\mathrm{vec}} 
\newcommand{\norm}[1]{\| #1 \|}
\newcommand{\inner}[2]{\langle #1, #2 \rangle}

\newcommand{\C}{\mathbb{C}}

\newcommand{\interior}{\mathrm{int}}

\newcommand{\lxy}{{\rm L}(\mathcal{X}, \mathcal{Y})}
\newcommand{\mcx}{\mathcal{X}}
\newcommand{\mcy}{\mathcal{Y}}

\newcommand{\corr}{\mathrm{Corr}}
\newcommand{\la}{\langle}
\newcommand{\ra}{\rangle}
\newcommand{\p}{\partial}

\newcommand{\R}{\mathbb{R}}
\def\be{\begin{equation}}
\def\ee{\end{equation}}
\def\bi{\begin{itemize}}
\def\ei{\end{itemize}}

\newcommand{\Herm}{\mathcal{S}}
\newcommand{\herm}{\Herm}
\newcommand{\Pos}{\mathcal{S}_+}
 \newcommand{\pos}{\Pos}

\newcommand{\dnn}{\mathcal{DNN}}

\newcommand{\gram}{\mathrm{Gram}}

\newcommand{\CPSD}{\mathcal{CS}_+}
\newcommand{\CP}{\mathcal{CP}}
\newcommand{\DNN}{\mathcal{DNN}}
\newcommand{\NN}{\mathcal{N}}
\newcommand{\NS}{\mathcal{NS}}
\newcommand{\intr}{\mathrm{int}}

\newcommand{\qvalue}{\omega_\mathcal{Q}(\calG)}
\newcommand{\cvalue}{\omega_\mathcal{C}(\calG)}
\newcommand{\nsvalue}{\omega_\mathcal{NS}(\calG)}
\newcommand{\unresvalue}{\omega_{\calP}(\calG)}

\newcommand{\kvalue}{\omega(\mathcal{K},\calG)}

\newcommand{\asbt}{N}
\newcommand{\sfT}{{\sf T}}
\newcommand{\xikvalue}{\xi(\mathcal{K}, \calG)}

\newcommand{\pvaluecpsd}{\omega(\CPSD,\calG)}
\newcommand{\dvaluecpsd}{\xi(\CPSD, \calG)}
\newcommand{\pvalueclcpsd}{\omega({\rm cl}(\CPSD),\calG)}

\newcommand{\rmx}{{\rm X}}

\newcommand{\pabst}{p=(p(a,b|s,t))}
\newcommand{\paass}{p=(p(a,a'|s,s'))}
\newcommand{\xasbt}{X[(s,a),(t,b)]}
\newcommand{\xsa}{\{X^s_a\}_{s\in S, a\in A}}
\newcommand{\ytb}{\{Y^t_b\}_{t\in T, b\in B}}
\newcommand{\sxsa}{\{x^s_a\}_{s\in S, a\in A}}
\newcommand{\sytb}{\{y^t_b\}_{t\in T, b\in B}}

\newcommand{\nso}{\mathcal{NSO}}
\newcommand{\anea}{a\ne a'\in A}
\newcommand{\mcv}{\mathcal{V}}
\newcommand{\mcw}{\mathcal{W}}
\newcommand{\mch}{\mathcal{H}}

\usepackage{color,graphicx}

\newtheorem{theorem}{Theorem}[section]
\newtheorem{corollary}[theorem]{Corollary}
\newtheorem{definition}[theorem]{Definition}
\newtheorem{proposition}[theorem]{Proposition}
\newtheorem{lemma}[theorem]{Lemma}
\newtheorem{remark}[theorem]{Remark}

\newtheorem{result}{Result}

\begin{document}

\title{\bf Linear conic formulations for two-party  correlations and values of nonlocal games}
\author[1]{Jamie Sikora\thanks{cqtjwjs@nus.edu.sg}}
\author[1,2]{Antonios Varvitsiotis\thanks{AVarvitsiotis@ntu.edu.sg}} 
\affil[1]{Centre for Quantum Technologies, National University of Singapore, and 

MajuLab, CNRS-UNS-NUS-NTU International
Joint Research Unit, UMI 3654, Singapore 
}

\affil[2]{School of Physical and Mathematical Sciences, Nanyang Technological University, Singapore}
\date{August 30, 2016}
\maketitle

\vspace{-5mm} 

\begin{abstract}
In this work we  study the sets of two-party correlations generated from   a Bell scenario  involving  two spatially separated systems  with respect to  various physical models. We show that the sets of classical, quantum, no-signaling and unrestricted  correlations  can be expressed as projections of affine sections of appropriate convex cones. 
As a by-product,  we identify a  spectrahedral outer approximation  to the set of quantum correlations  which is contained in the first level  of the Navascu\'es, Pironio  and  Ac\'in (NPA) hierarchy and also  a sufficient condition for the set of quantum correlations to be closed.   Furthermore, by our conic  formulations,  the value of a nonlocal game over the sets of classical, quantum, no-signaling and unrestricted correlations can be cast  as a  linear conic program. This allows us to show that a  semidefinite programming upper bound to the classical value of a nonlocal game introduced by Feige and Lov\'asz is in fact an upper bound to the quantum value of the game and moreover, it is at least as strong as optimizing over the first level of the NPA hierarchy.  
Lastly, we show that deciding the existence of a perfect quantum (resp. classical) strategy  is equivalent to deciding the feasibility of a linear  conic program over the cone of completely positive semidefinite matrices (resp. completely positive matrices). By specializing the results to  synchronous nonlocal games, we recover the  conic formulations for various  quantum and classical  graph parameters that were recently derived  in the literature.
\end{abstract}  

\smallskip

 
\newpage  
\section{Introduction} 
Consider two parties, Alice and Bob,   {who} individually perform  measurements on a shared physical system {without communicating}.  A problem  of fundamental importance with which we are primarily concerned in this work is to characterize the structure of  the sets of correlations that can arise    between Alice and Bob, with respect to various physical  models. 
 
In one of the most celebrated discoveries of modern physics John  Bell  showed that quantum mechanical systems can exhibit correlations  that cannot be reproduced  within the framework of classical physics  \cite{B64}. This  fact  has  received extensive   experimental verification, see  \cite{PhysRevLett.28.938,AGG82} for examples. 
In addition to {their theoretical significance, these  correlations} 
have been increasingly regarded  as a valuable resource for distributed tasks such as  unconditionally secure cryptography \cite{Ekert91}  and  randomness certification~\cite{Colbeck09} among others.

In order to tackle this problem, {we take}  
the viewpoint of linear  conic optimization. Specifically, we introduce the notion of conic correlations and  show that the sets of classical, quantum, no-signaling and {unrestricted} correlations  can be expressed as conic correlations over appropriate convex cones. Consequently,  conic correlations provide us with a unified framework where we can study the properties of many interesting families of correlations.
Furthermore, using our conic characterizations, we can express the classical, quantum, no-signaling and {unrestricted} values of a nonlocal game as linear conic programs. This allows one  to use  the  arsenal     of linear conic programming theory  in order to study how the various values of a nonlocal game  relate to each other and to better understand their properties. 

There exists a  significant body of work addressing these   
questions from a mathematical  optimization perspective. In the celebrated work \cite{NPA08} Navascu\'es, Pironio  and  Ac\'in constructed  a hierarchy of  
spectrahedral outer approximations to the 
set of quantum correlations.
Another  fundamental result is that the quantum value  
{of  an XOR nonlocal game}  is given by a semidefinite program \cite{TS87,Cleve04}.  Furthermore, the quantum value of a 
{unique nonlocal game} can be tightly  approximated using semidefinite programming~\cite{KRT}.  Lastly, mathematical optimization has also proven to be extremely useful  for (classical and quantum) parallel repetition results \cite{FL92,CSUU08,KRT,DS14,DSV14}. 
 
\paragraph{{\bf Two-party correlations.}} 
 
  Consider the following thought experiment: Two spatially separated parties, Alice and Bob, perform measurements  on some  shared  physical system.  Alice has a set of possible   measurements at her disposal, where each measurement is labeled by some element of a finite set $S$.  The set of possible outcomes of each of Alice's measurements is  labeled  by the elements of  some finite set $A$. Similarly,  Bob  has a set $T$  of possible   measurements at his  disposal each with  possible outcomes labeled by the elements of some finite set $B$. 
 Note that  we use the term ``measurement'' very loosely at this point as the details  depend on the underlying physical theory.  We refer to a thought experiment  as described above as  a {\em Bell scenario}. 
 
 At each run of the experiment   Alice and Bob  without communicating 
 choose measurements $s \in  S$ and $t\in T$ respectively which they use to  measure their individual systems.  Following the measurement they get $a\in A$  and $b\in B$ as outcomes. Since the  measurement process is   probabilistic, each time the experiment  is conducted Alice and Bob  might generate  different outcomes.  The Bell scenario  is completely described by the  joint conditional probability distribution  $\pabst_{a,b,s,t}$, where $p(a,b|s,t)$ denotes the  conditional probability  that upon performing measurements $s\in S$ and $t\in T$,  Alice and Bob  get  outcomes $a\in A$ and $b\in B$, respectively.  
 
{For any Bell scenario, 
the set  of all  
joint probability distributions, denoted by $\mathcal{P}$, consists of all vectors $ \pabst\in \R^{|A\times B \times S \times T|}$ that satisfy 
$p(a,b|s,t) \ge~0,$ for all $a,b,s,t$, and $\sum_{a \in A, b \in B}  p(a,b|s,t) =1,$ for all $s,t$. The elements of $\mathcal{P}$ are called {\em correlation vectors} or simply~\emph{correlations}.} 

{A question of fundamental theoretical  interest  is to describe  the  correlations that  can arise 
within a Bell scenario 
as described above  with respect to  various physical models. We now   briefly introduce the models and the corresponding sets of correlations that are relevant to this work.  For additional details the reader is referred to the extensive survey \cite{Brunner14} and  references therein.}

\paragraph{Classical correlations.} 
A {\em classical strategy} allows  Alice and Bob to determine their outputs by employing  both private and shared randomness. Formally, a classical strategy is given by:
 \bi
 \item[$(i)$]A  shared  random variable {$i$}  with domain 
$[n]$, each sample occurring with probability~$k_i$.
 \item[$(ii)$] For each $i\in [n]$ and $ s\in S$   a probability 
distribution $\{x_{a}^{s, i}: a\in A\}$.
\item[$(iii)$]  For each $i\in [n]$ and $t\in T$  a probability 
distribution  $\{y_{b}^{t, i}: b\in B\}$.
 \ei 
Given that the value  of the shared randomness is  $i \in [n]$, if   Alice chooses  measurement  $s \in S$  she determines her output $a\in A$ by sampling from  the distribution $\{x_{a}^{s, i}\}_{a\in A}$. 
Bob acts analogously and determines his output by sampling from the distribution $\{y_{b}^{t, i}\}_{b\in B}$. 
{ 
{Formally, a correlation $p\in \calP$ is called  {\em classical} if  there exist {\em nonnegative} scalars $\{k_i\}_i$, $ \{x_{a}^{s,i}\}_{a,s,i}$, $\{y_{b}^{t,i}\}_{b,t,i}$ satisfying $\sum_{i \in [n]} k_i = 1$, \   
{$\sum_{a \in A} x_{a}^{s,i} =\sum_{b \in B} y_{b}^{t,i} = 1$}, for all $s,t,i$} and   
\be \label{classicalcorrelations}
 p(a,b|s,t)=\sum_{i=1}^n k_ix_a^{s,i}y_b^{t,i}, \text{ for all } a,b,s,t.  
\ee
{We denote the set of classical correlations by $\mathcal{C}$. Note that in the literature, classical correlations are also referred to as ``local''  and denoted by~$\mathcal{L}$.}

The set of classical correlations  forms a convex polytope in $\R^{|A\times B\times S\times T|}$. Its vertices correspond  to deterministic strategies, i.e.,  correlations of the form  $p(a,b|s,t)=\delta_{a,\alpha(s)}\delta_{b,\beta(t)} $ {for some pair of functions  $\alpha: S \rightarrow~A$ and $\beta: T\rightarrow B$, where   $\delta_{i,j}$ denotes the  Kronecker delta function. } 
 
\paragraph{Quantum correlations.}
A {\em quantum strategy} for a Bell scenario allows Alice and Bob to determine their outputs  by performing measurements  on a shared {quantum state} {(the reader is referred to  Section~\ref{sec:qmechanics} for  background on quantum information and the context behind the mathematical formalism in the following discussion.)}  A correlation $p\in \calP$ is called {\em quantum}  if   there exists finite dimensional complex Euclidean spaces  $\mcx$ and $\mcy$, a unit vector $\psi \in \mcx\otimes \mcy$, Hermitian {positive semidefinite (psd)} operators $\{ X^s_a \}_{a \in A }$ satisfying $\sum_{a} X^s_a=\mathbb{I}_\mcx,$ for each $s \in S$ and Hermitian psd operators $\{ Y^t_b\}_{b \in B}$ satisfying $\sum_{b} Y^t_b=\mathbb{I}_\mcy$, for each $t \in T$, {where}
\be\label{eq:quantumcorrelations}
p(a,b|s,t) = \psi^* (X_a^s \otimes Y_b^t) \psi, \text{ for all } a,b,s,t.
\ee
We denote the set of quantum correlations by 
$\mathcal{Q}$.
 
The set of quantum correlations is a {non-polyhedral set 
whose  structure}  has been extensively studied  but is nevertheless not well understood (e.g. see~\cite{Brunner14}). In particular, it is not even known whether $\calQ$   
is 
closed. On the positive side, Navascu\'es, Pironio, and Ac\'in (NPA)  in  \cite{NPA08} identified a hierarchy of spectrahedral outer approximations to the set of quantum correlations. {Although the NPA  hierarchy converges, it is not known  whether it converges to the set of quantum correlations.}
 
\paragraph{No-signaling correlations.}
A correlation $p\in \calP $ is {\em no-signaling} if Alice's local marginal probabilities  are independent  of Bob's choice of measurement and, symmetrically, Bob's local marginal probabilities are independent   of Alice's choice of measurement. {Algebraically,    
{$p\in \calP$  is  {\em no-signaling} if it satisfies:}}
\be\label{ns1}
 \sum_{b\in B} p(a,b|s,t) = \sum_{b\in B} p(a,b|s,t'), \  \text{ for all } s\in S, t\ne t'\in T, \text{ and }
 \ee
\be\label{ns2}
 \sum_{a\in A} p(a,b|s,t) = \sum_{a\in A} p(a,b|s',t),  \ \text{ for all } t\in T, s\ne s'\in S. 
 \ee 
We denote the set of no-signaling correlations by $\NS$.}  

{The no-signaling conditions \eqref{ns1} and \eqref{ns2}  are}  a natural physical requirement since if  they are violated   at least one party  can receive information about the other party's input \emph{instantaneously}, contradicting the fact that  information cannot travel faster than the speed of light. 
   
\medskip  
  
It 
{is immediate  from  physical context} 
that every classical correlation is also quantum (cf. Theorem \ref{thm:ccorrelationvector}). Furthermore, it is easy to verify that every quantum correlation is no-signaling (cf. Theorem \ref{thm:corrktoNS}). On the other hand, it is well-known that there exist  quantum correlations that are not classical and no-signaling correlations that are not quantum.  
 In other words, we have that 
\be\label{cor:inclusions}
\mathcal{C} \subsetneq \calQ \subsetneq \NS \subsetneq \calP,
\ee
and in this paper we give (alternative) algebraic proofs of these containments. 

\paragraph{{\bf Two-player one-round nonlocal games.}}
\label{sec:2player1round}

As we  mentioned, the set of quantum correlations is a strict superset of the set of   classical correlations. 
 How can we  identify quantum correlations that are not classical? One approach    is  via the framework of nonlocal games which we now~introduce. 
 
A  nonlocal game is a thought experiment between two spatially separated parties, Alice and Bob, who can only communicate with a third party, a referee, who  decides whether they win or lose.  Formally, a  (two-player one-round) {\em nonlocal game}   is specified by four finite sets $A,B,S,T$, a  probability distribution $\pi$ on $S\times T$ and a Boolean predicate $V: A\times B\times S\times T \rightarrow  \{0,1\}$. We denote {the nonlocal game}  by $\calG(\pi, V)$ or simply $\calG$ when there is no need to specify  $\pi$ and~$V$.  
 
The {nonlocal} game $\calG(\pi, V)$ proceeds as follows: The referee using the distribution $\pi$ samples a pair of questions $(s,t) \in S\times T$ and sends $s$ to Alice and $t$ to Bob. After receiving their questions, Alice and Bob use some  strategy to determine their  answers $a\in A$ and $b\in B$ which they send back to the referee.   The players {\em win}  the game if $V(a,b|s,t)=1$ and they {\em lose} otherwise.  
 
The objective of the  players is to maximize their probability of winning the game.
To do this the  players are not allowed to communicate after they  receive their questions  but they can  agree on some common strategy before the start  of the game using their knowledge of $V$ and $\pi$.

Fix a particular strategy for the game that gives rise to the correlation $\pabst\in \calP$. 
The probability that Alice and Bob win the game  using this strategy is given by 
\[ 
\sum_{s \in S} \sum_{t \in T} \pi(s,t) \sum_{a \in A} \sum_{b \in B} V(a,b|s,t) p(a,b|s,t). 
\] 
For  a fixed  set of correlations  $\calS \subseteq \calP$ we denote by $\omega_\calS(\calG)$ the maximum probability  Alice and Bob can win the game  $\calG$ when they use strategies that generate correlations that lie  in $\calS$. Formally:
\be\label{Svalue}
\omega_\calS(\calG) := \sup \left\{ 
\sum_{s \in S} \sum_{t \in T} \pi(s,t) \sum_{a \in A} \sum_{b \in B} V(a,b|s,t) p(a,b|s,t) : p \in \calS 
\right\}.  
\ee

In this paper we restrict our attention to \textit{(i)} the \emph{classical value} denoted $\omega_{\mathcal{C}}(\calG)$, \textit{(ii)} the \emph{quantum value}  denoted $\omega_{\calQ}(\calG)$, \textit{(iii)}  the \emph{no-signaling value} denoted $\omega_{\NS}(\calG)$ and \textit{(iv)} the \emph{unrestricted value} denoted  $\omega_{\calP}(\calG)$. 
As an immediate consequence of the set inclusions given in Equation~\eqref{cor:inclusions} we have  
\[ 
\omega_{\mathcal{C}}(\calG) 
\leq 
\omega_{\mathcal{Q}}(\calG) 
\leq 
\omega_{\mathcal{NS}}(\calG) 
\leq 
\omega_{\mathcal{P}}(\calG),
\] 
for any nonlocal game $\calG$.

{As a concrete example of the above definitions we now describe the CHSH game~\cite{CHSH69}.   
This game has $A=B=S=T=\{0,1\}$, $\pi$ is uniform, and $V(a,b|s,t)=1$ if and only if $a\oplus b=s \cdot  t$, where $\oplus$ denotes  addition modulo~2. Informally, the referee sends a random bit $s$ to Alice and an independently random bit $t$ to Bob. The players respond with single bits $a$ and $b$, respectively. Alice and Bob win if $V(a, b|s, t) = 1$, i.e., if $a \oplus b$ is equal to the logical AND of their questions.  
It is well-known  {that the no-signaling value of the CHSH game  is $1$, the quantum value is $\cos^2(\pi/8)\approx 0.85$ and the classical value is~$3/4$}.} 
  
\paragraph{{\bf Convex cones of interest.}}
 
Consider a vector space $\mcv$ endowed with an inner product $\la \cdot, \cdot\ra$. The {\em Gram matrix} of {the} vectors $\{x_i\}_{i=1}^n\subseteq \mcv$, denoted by 
\begin{equation} 
{\gram(\{x_i\}_{i=1}^n),} 
\label{eqn:gram} 
\end{equation} 
is the  $n\times n$ matrix whose $(i,j)$ entry is given by $\la x_i,x_j\ra$.  
{We say  that the vectors $\{x_i\}_{i=1}^n$ form a {\em Gram representation} of $X=\gram(\{x_i\}_{i=1}^n) $.} 

We denote by $\mathcal{S}^n$ the set of $n \times n$ real symmetric matrices which we equip with the Hilbert-Schmidt inner product $\la X,Y\ra:=\tr(XY)$. 
A matrix $X\in \calS^n$ is called {\em positive semidefinite} (psd)  if $X=\gram(\{x_i\}_{i=1}^n)$ for some family of  real  vectors  $\{x_i\}_{i=1}^n\subseteq \R^d$ (for some $d\ge 1$). {Equivalently, a matrix is psd if and only if its eigenvalues are nonnegative.}  {A nonsingular psd matrix}  is   called {\em positive definite}. We denote by $\calS^n_+$  (resp. $\calS^n_{++}$)  the set of $n\times n$ psd matrices (resp. positive definite matrices). 
The set $\calS^n_+$ forms a closed, convex, self-dual cone  whose structure is well understood (e.g. see \cite{Barvinok} and references therein).  Linear optimization over $\calS^n_+$ is called {\em semidefinite programming} (SDP) and  its {optimal} value can be approximated within arbitrary precision in polynomial time  using the ellipsoid method, under reasonable assumptions (e.g. see  \cite{BTN13}). 

{The {\em nonnegative cone}, denoted by $\NN^n$, consists  of the $n\times n$ entrywise nonnegative matrices  in  $\calS^n$. It is easy to verify  that $\NN^n$ is  a self-dual cone.}

{A matrix is called {\em doubly nonnegative} if its psd and entrywise nonnegative. We denote by $\DNN^n$ the set of   $n\times n$ doubly nonnegative  matrices,~i.e., 
\be\label{nonegative}
\DNN^n := \left\{ X \in \herm_+^n : X_{i,j} \geq 0,  \ \textup{ for all }\  1\le i,j\le n \right\}, 
\ee
which is  known  to form  a full-dimensional  closed convex cone.}
  
{A matrix  $X\in \calS^n$ is called {\em completely positive} if}
\be\label{eq:cpcone} 
{X=\gram(\{x_i\}_{i=1}^n), \text{ where } \{x_i\}_{i=1}^n\subseteq  \R^d_+ \  (\text{for some }  d\ge 1).}
\ee 
{The set of   $n\times n$ completely positive matrices forms a full-dimensional  closed convex cone known as the {\em completely positive cone},  and is    denoted by  $\CP^n$.} The  structure of the $\CP$ cone  has been extensively studied, e.g.  see \cite{BSM03}. Optimization over $\CP$ is intractable since there exist NP-hard combinatorial optimization problems that  can be formulated  as linear optimization problems over $\CP$ \cite{KP02} (see also Section~\ref{sect:perfect}). On the positive side, there exist SDP hierarchies that can be used to  approximate $\CP$ from  the interior \cite{Las12} and from the exterior \cite{Par00}.     
  
Thinking of  nonnegative vectors as diagonal psd  matrices suggests  a natural generalization of the completely positive cone. A matrix $X\in \calS^n$ is called {\em completely positive semidefinite} (cpsd) 
if 
\be
{X=\gram(\{X_i\}_{i=1}^n), \text{ where }\{X_i\}_{i=1}^n\subseteq  \pos^d \ (\text{for some } d\ge 1).}
\ee  
 The set of $n \times n$ {cpsd matrices} 
 forms a full-dimensional convex cone  denoted  by $\CPSD^n$. The $\CPSD$ cone   was introduced recently   as a tool to provide conic programming formulations for the quantum chromatic number of a  graph \cite{LP14} and quantum graph homomorphisms \cite{R14b} (cf.  Section~\ref{sec:perfectcorrelated}).    Nevertheless,  
 its structure appears  to be very   complicated. In particular it is not known whether $\CPSD$ forms  a closed set~\cite{BLP15}. Furthermore, given a matrix $X\in \CPSD^n$,  no upper bound is known on the size of the psd matrices in a Gram representation  for~$X$.  {This is in  contrast to the completely positive cone, where we can always find a Gram representation satisfying \eqref{eq:cpcone} using nonnegative vectors whose dimension is at most quadratic in the size of the matrix.} 
   Lastly,   combining 
 results from \cite{LP14} and~\cite{Ji13} it follows that linear optimization over  $\CPSD$ is NP-hard.
 
 {It  is immediate 
 from the definitions given above that for all  $n\ge 1$~{we have}}
 \be\label{coneinclusions}
  \CP^n \subseteq \CPSD^n \subseteq \DNN^n.
 \ee
For  $n\le 4$ it is known that   $\CP^n= \DNN^n$ \cite{MM61}.  
On the other hand, for $n\ge 5$ all the inclusions   in~\eqref{coneinclusions} are known to be strict \cite{LP14}.  
 
As the mathematical formulation of quantum mechanics  is stated in terms of psd  matrices with complex entries, in some parts of this work we consider  matrices with complex entries. {We denote by $\mch^n$ the set of $n\times n$ Hermitian matrices.  A matrix $X\in \mch^n$ is called {\em Hermitian positive semidefinite} if $z^*Xz\ge~0$, for all $z\in \C^n$.
We denote by  $\mch^n_{+}$ (resp. $\mch^n_{++}$)  the  set of $n\times n $ Hermitian positive semidefinite  matrices  (resp. $n \times n$ Hermitian positive definite matrices).}
Occasionally, we also use the notation $\mch_+(\mcx)$ to denote the positive operators acting on a finite dimensional complex  Euclidean space $\mcx$. 
For a matrix $X\in \mch^n$ we  write $X=\mathcal{R}(X)+i\mathcal{I}(X),$ where {$\mathcal{R}(X)$} is the real part and $\mathcal{I}(X)$ is the  imaginary part of $X$.  If $X$ is Hermitian we get that $\mathcal{R}(X)$ is real symmetric and $\mathcal{I}(X)$ is real skew-symmetric.  Moreover,   for $X,Y\in \mch^n$ we have   $\la X,Y\ra=\tr\big(\mathcal{R}(X)\mathcal{R}(Y)-\mathcal{I}(X)\mathcal{I}(Y)\big)$. 
For a  matrix $X\in \C^{n\times n}$,~set 
 \be\label{eq:comtorealpsd}
T(X) := \dfrac{1}{\sqrt 2} \begin{pmatrix} \mathcal{R}(X) & -\mathcal{I}(X)\\ \mathcal{I}(X)&  \mathcal{R}(X)\end{pmatrix},
\ee
and notice that $T$ is a bijection   between complex $n\times n$ matrices  and  real $ 2n\times 2n$ matrices. More importantly,  we have that  $X\in \mch^n_+$  if and only if $T(X)\in \mathcal{S}^{2n}_+$  and moreover  $\inner{X}{Y} = \inner{T(X)}{T(Y)}$ for all $X,Y\in \mch^n_+$. This shows that the set of  cpsd matrices does not change if we allow the psd matrices in the Gram decompositions to be Hermitian psd instead of just real {psd}.  
 
Lastly, a symmetric matrix $X \in \NN^N$ (where {$N:= |S||A| + |T||B|$}) is called {\em no-signaling}, {denoted by $\nso$}, if it satisfies 
$$\sum_{a \in A} \xasbt =  
\sum_{a \in A} X[(s',a),(t,b)], \ \forall  b \in B, t \in T, s\ne s' \in S, \text{ and }$$
$$\sum_{b \in B} \xasbt =  
\sum_{b \in B} X[(s,a),(t',b)], \ \forall  a \in A, s \in S, t\ne t' \in T.$$ 
 
\paragraph{{\bf {Contributions.}}}
 
Consider a Bell scenario  with question sets $S,T$ and answer sets $A,B$  
{and set $N:= |S||A| + |T||B|$.} 
In this work we  mostly consider   symmetric  $N\times N$ matrices. The rows and columns of such a matrix  are each indexed by $S \times A$ and $T \times B$ and it is useful  to think of $X$     as being partitioned into blocks $X_{i,j}$, where  each block   is indexed by a  pair  of questions $i,j \in S\cup T$. The size  of each  block is  $(i)$ $|A|\times |A| $ if  $i,j\in S$, $(ii)$ $|A|\times |B|$ if $i\in S,j\in T$, $(iii)$ $|B|\times |A|$ if $i\in T,j\in S$ and $(iv)$ $|B|\times |B|$ if~{$i,j\in T$}. 

{For $i,j\in S\cup T$ we define $J_{i,j}\in \calS^N$ to be the matrix which  acts  on  a matrix $X\in \calS^N$ by summing all entries in block~$X_{i,j}$, i.e.,
 \be\label{eq:jmatrix}
 \la J_{i,j},X\ra=\sum_{k,l} X[(i,k),(j,l)].
 \ee}
{At times, we also consider symmetric $(N+1)\times (N+1)$  matrices, {which have an extra row and column indexed by ``$0$''.}
We also extend the operator $J_{i,j}$ {defined in} \eqref{eq:jmatrix} to act on $\calS^{N+1},$ where we define 
 \be\label{eq:jmatrix2}
 \la J_{0,i},X\ra=\sum_k X[0,(i,k)], \text{ for all } i\in S\cup T {\text{ and }  \la J_{0,0},X\ra = X[0,0].}
 \ee
{Lastly, in the final  part of this work we also consider matrices in  $\calS^{|S\times A|}$. In this case, for any $s,s'\in S$  we  denote by $J_{s,s'}$  the operator that acts on $X\in \calS^{|S\times A|}$ by summing the entries in block $X_{s,s'}$.}
 
For brevity, we do   not specify in the notation whether the operator  $J_{i,j}$ acts on $\calS^{N}, \calS^{1+N}$ or $\calS^{|S\times A|}$ as this  {is always clear from context}.}
   
\medskip 

\noindent{\em Correlation sets.}
In Section \ref{sec:correlations} we express  the sets of classical, quantum, no-signaling and unrestricted correlations as projections of affine slices of appropriate convex cones.
To achieve this {we use the following definition}.  
  
\begin{definition}\label{def:kcorr}
For a convex cone  $\calK \subseteq \NN^N$ the set of $\calK$-{\em correlations}, {denoted $\corr(\calK)$}, is defined as 
the set  of vectors $ \pabst\in \R^{|A\times B \times S \times T|}$  for which there exists a matrix   $X\in \calK $ satisfying: 
\begin{equation}\label{kcorrelations}
\begin{aligned}
 \inner{J_{i,j}}{X} &  = 1,\  \text{for  all } i,j \in S\cup T, \text{ and } \\
 \xasbt & = p(a,b|s,t),\   \text{for all } a,b,s,t. 
\end{aligned}
\end{equation} 
\end{definition} 
 
 In Theorem \ref{thm:correlations} we show there  exist appropriate choices of convex cones $\calK$ for which the sets of $\calK$-correlations capture the sets of classical, quantum, no-signaling and unrestricted correlations. Specifically: 

\begin{result} \label{res:correlations}
Consider an arbitrary vector  $\pabst\in \R^{|A\times B \times S \times T|}$. Then,
\begin{itemize}
\item[(i)]  $p$ is a classical  correlation (i.e., $p \in \calC$) if and only if $p\in \corr(\CP)$. 
\item[(ii)]  $p$ is quantum correlation (i.e., $p \in \calQ$) if and only if $p\in \corr(\CPSD)$.
\item[(iii)]  $p$ is a no-signaling correlation (i.e., $p \in \NS$)  if and only if $p\in \corr(\nso)$.
\item[(iv)]  $p$ is a correlation (i.e., $p \in \calP$)  if and only if $p\in \corr(\NN)$.
\end{itemize}
\end{result}
 
{We note that} upon completion of this work we found   that a  result similar to Result~\ref{res:correlations} \textit{(ii)} has been derived  independently  in the unpublished note  \cite{MR14b}.

{The use of {convex} cones to characterize the sets of quantum, classical, and no-signaling correlations was also  an essential ingredient
 in~\cite{Fr12b}.}  
   
As {suggested} by  Result \ref{res:correlations}  
 the  notion of conic correlations   provides    a general  framework allowing us  to  phrase 
and study the properties of many interesting sets  of correlations.  Notice that  whenever $\calK_1\subseteq \calK_2$ we have that $\corr(\calK_1)\subseteq \corr(\calK_2)$. Consequently,  {the inclusions from} \eqref{coneinclusions} combined with  Result~\ref{res:correlations} {imply} that  $\mathcal{C} \subseteq \calQ \subseteq \calP.$
As already mentioned in the introduction  these inclusions are well-known  but the notion of  conic correlations allows us to recover them  within a purely mathematical~framework.  
 
{Note  that when $\calK$ is a    {\em closed} convex cone the set of  $\calK$-correlations is also closed. Furthermore, recall that  the set of quantum correlations is not known to be closed. As an immediate consequence of Result~\ref{res:correlations}~$\textit{(ii)}$ it follows that if the $\CPSD$ cone is closed then 
the set of quantum correlations is also closed (cf. Proposition \ref{cor:closedness}).  
The  same  observation was  made  independently in the unpublished note~\cite{MR14b}.}
To the best of our knowledge, the first work where  the structure of the closure of the set of quantum correlations was studied is~\cite{Fr12}.
     
{In Theorem \ref{thm:corrktoNS} we show that   for any  $\calK\subseteq \DNN$ we have   $\corr(\calK)\subseteq~\NS$. This fact combined with Result \ref{res:correlations} and   the inclusion  $\CPSD \subseteq \DNN $ implies  that} 
\be\label{vddfgergetrg}
 \calQ\subseteq \corr(\DNN)\subseteq \NS,
\ee i.e., $\corr(\DNN)$ forms a  spectrahedral outer approximation for the set of quantum correlations  which is contained in  the set of no-signaling correlations.  {In Theorem~\ref{thm:corrnpa} we compare $\corr(\DNN)$ with the first level of the NPA hierarchy, denoted  by ${\rm NPA}^{(1)}$.} We are able to show the following:

\begin{result}\label{NPA}
For any Bell scenario we have that $\corr(\DNN)\subseteq {\rm NPA}^{(1)}$.
\end{result}

\noindent{\em Game values.}
In Section \ref{sec:conicprogrammingformulations} we study the value of a nonlocal game when   the players  use strategies that generate classical, quantum, no-signaling or unrestricted correlations. 
To state our results in a succinct manner  we introduce some notation that  is used  throughout the paper. 
The {\em cost matrix} of a game $\calG(\pi,V)$ is  the {$|S \times A|$ by $|T \times B|$}  matrix  $C$ whose entries are given~by 
\be\label{hatc}
C[(s,a),(t,b)] := \pi(s,t) V(a,b|s,t), \  \text{ for all } a\in A,b\in B,s\in S,t\in T.
\ee  
The {\em symmetric cost matrix} of the game  $\calG$ is the $\asbt\times N$  matrix  
\be\label{symmetriccostmatrix}
\hat{C} :=    
\dfrac{1}{2}
\begin{pmatrix} 
0 & C \\
C^\sfT & 0 
\end{pmatrix}.
\ee
For a convex cone $\calK\subseteq \NN^N$  we {denote by} $ \omega(\calK, \calG)$ 
 the maximum  success probability of winning      $\calG$ when the players  use strategies that generate  $\calK$-correlations, i.e., 
\be\label{kconic}\tag{$\mathcal{P}_{\calK}$} 
  \omega(\calK, \calG) := \text{sup}\{ \inner{\hat{C}}{X}:  \la J_{i,j},X\ra=1, \text{ for all }\  i,j\in S\cup T,
\ X \in  \calK\}.
\ee  
 
Note  that \eqref{generalconic} is an instance of a linear conic program over the convex cone $\calK$. As an immediate consequence  of Result \ref{res:correlations}, the classical, quantum, no-signaling and unrestricted {values} of a nonlocal game can all be expressed as linear conic programs over appropriate convex cones.  
   
Having established conic formulations for $\cvalue$ and  $\qvalue$   we also    study  the corresponding  dual conic  programs and their properties {in Section~\ref{sec:conicprogrammingformulations}}.  
Furthermore, we use  our  formulations  to  compare the various values of a nonlocal game. For this, let  ${\rm SDP}^{(1)}(\calG)$ denote    the value of the  SDP obtained by optimizing over ${\rm NPA}^{(1)}$, i.e., the first level of the NPA hierarchy.  In Proposition~\ref{sdpupperbound} we~show: 
\begin{result}For any game $\calG$ we have $\qvalue \le \omega(\DNN,\calG) \le {\rm SDP}^{(1)}(\calG).$
\end{result}
 
Interestingly, $\omega(\DNN,\calG)$ was already introduced by Feige and Lov\'asz as an SDP upper bound to the classical value 
of a nonlocal game
\cite{FL92}. 

{In a very recent and independent work, a similar observation was also made by creating a new SDP hierarchy approximating 
the quantum value of a nonlocal game~\cite{BFS15}. The first level of that  hierarchy corresponds to $\omega(\DNN,\calG)$.}

Lastly,  we use our conic formulations  
to study  the problem of deciding the existence of a strategy that wins a nonlocal game with certainty.
\begin{definition}\label{perfectkstrategy}
Consider a nonlocal game $\calG(\pi, V)$ and a convex cone  $\calK \subseteq~\NN$. {We say  that $\calG$ admits  a {\em perfect $\calK$-strategy} if $\kvalue=1$ and moreover, this value is achieved by some correlation in $\corr(\calK)$.} 
\end{definition}
 
 We show that deciding the existence of a perfect $\calK$-strategy is equivalent to the feasibility of a linear conic program over $\calK$. This fact combined with   Result \ref{res:correlations} implies  that deciding the existence  of a perfect classical, quantum, no-signaling and unrestricted strategy is equivalent to  the feasibility of a conic program over the cones $\CP, \CPSD,\nso$ and $\NN$, respectively (cf. {Corollary}~\ref{conicformulation_perfectstrategies}).

 It is well-known  that deciding the existence of a perfect classical strategy  is $ {\rm NP}$-hard (see also  Section \ref{sec:perfectcorrelated}).
 Furthermore, it was recently shown that deciding the existence of a perfect quantum strategy is also ${\rm NP}$-hard  \cite{Ji13}. Nevertheless, this problem is currently not known to be decidable. Note that our reformulation as a conic feasibility program does not render the problem decidable as no algorithms are  known for determining  the feasibility of a $\CPSD$-~program.
  
 \medskip  
In Section \ref{sect:corr} we restrict to  Bell scenarios where  $S=T$ and $A=B$.   We first specialize our conic characterizations from Section \ref{sec:correlations} to  {\em synchronous} correlations, i.e.,  correlations with the property that whenever the players receive the same question they need to respond with the same answer. 
  
  Recently there has been  interest in the study of  synchronous correlations as they {correspond to} perfect strategies  for   graph homomorphism {games} and more generally,  synchronous nonlocal  games  (e.g. \cite{Paulsen14,MRV14,DP15}).
{Another characterization of the set of  synchronous quantum  correlations   in terms of the existence of a  $C^*$-algebra with certain properties was given in \cite{Paulsen14}. Furthermore, it was  shown in \cite{DP15} that Connes' embedding conjecture is  equivalent to showing that two families of sets
 of quantum synchronous correlations coincide.}
   
  Furthermore, in Section \ref{sect:corr}  we study {\em synchronous} nonlocal games, i.e., games  where    both  players share the same question and answer sets and in order to win, whenever they    receive the same question they have to respond with the same answer   (e.g. \cite{Cameron07,LP14,R14b,ML15}).  
 {The notion of synchronous games was implicit in \cite{Paulsen14} and was formally defined in \cite{MRV14} and \cite{DP15}. 
 We 
 focus on  the problem of deciding whether a synchronous game admits a perfect classical or quantum strategy.
 Synchronous games have the property that  perfect strategies   generate synchronous  correlations 
 (e.g. {see} \cite{MRV14})}.  
 In Theorem \ref{thm:conicformperfectstrategies} we show  that this problem is equivalent to the feasibility 
of a conic program with matrix variables of size~$|S\times A|$. 
Specializing   this 
 to graph homomorphism and graph coloring games we recover in a  uniform manner the conic formulations for quantum graph homomorphisms, the quantum chromatic number and the quantum independence number that were recently derived in the literature~\cite{LP14,R14b}. 
 
\paragraph{{\bf Paper organization.}} 

{In Section~\ref{sect:prelims} we introduce the notation and background on linear algebra, quantum mechanics, and linear conic programming needed for this work. In Section~\ref{sec:correlations} we discuss how correlations corresponding to various physical models can be represented as projections of affine slices of appropriate convex cones and identify  a spectrahedral outer approximation for the set of quantum correlations. In Section~\ref{sec:conicprogrammingformulations} we show that  values of nonlocal games can be formulated as conic programming problems and we further discuss the Feige-Lov\'asz SDP relaxation for the value of a nonlocal game. Additionally, we show that deciding the existence of a perfect strategy is equivalent to a conic feasibility problem. 
Finally, in Section~\ref{sect:corr} we specialize our characterizations to  synchronous correlations and synchronous game values. 
  
\section{Notation and background}  
\label{sect:prelims}
 
\paragraph{{\bf Linear algebra.}} 

A  {\em finite dimensional complex Euclidean space} refers   to the vector space~$\mathbb{C}^n$  (for some $n\ge 1$) equipped with the canonical inner product on $\mathbb{C}^n$.
We denote by $\{e_i\}_{i=1}^n$ the standard orthonormal basis of~$\C^n$, {and by $e$ the vector of {all $1$'s} of appropriate dimension}.
Given two   complex Euclidean spaces  $\mcx, \mcy$ we  denote by $\lxy$ the space of linear operators from $\mcx$ to $\mcy$ which  we endow  with the Hilbert-Schmidt inner product  
$\la X,Y\ra:=\tr(X^*Y)$ for $X,Y\in \lxy$. {For an operator $X\in \lxy$ we denote its {\em adjoint} operator by $X^\ast\in \mathcal{L}(\mcy,\mcx)$ and its {\em transpose} by $X^\sfT\in \mathcal{L}(\mcy,\mcx)$.}
We use  the  correspondence between $\lxy$ and $\mcy\otimes \mcx$ given by the  map
$\vecc: \lxy \rightarrow  \mcy\otimes \mcx,$ 
which is given by 
$\vecc(e_ie_j^*)=e_i\otimes e_j, $
on basis vectors and is extended linearly. The $\vecc(\cdot)$ map  is a linear  bijection  between $\lxy$ and $ \mcy\otimes \mcx$ and furthermore it is  an   isometry, i.e., $\inner{X}{Y} = \inner{\vecc(Y)}{\vecc(X)}$ for all $X,Y\in \lxy$. We make repeated use of the fact \begin{equation} \label{vecprop}
 \vecc(W)^*(X \otimes Y) \vecc(Z)= \vecc(W)^*\vecc(XZY^\sfT)=
\la W,XZY^\sfT\ra, 
\end{equation} 
 for Hermitian operators $X,Y,Z,W$  of the appropriate size (e.g. see \cite{Watrous}). 
  
Any  vector $\psi\in \mcy\otimes \mcx$ can be {uniquely} expressed as ${\psi=\sum_{i=1}^d\lambda_i \, y_i\otimes x_i} $ for 
some integer $d\ge 1$, positive  scalars $\{ \lambda_i\}_{i=1}^d$, and orthonormal sets ${\{ y_i\}_{i=1}^d\subseteq \mcy}$ and $\{ x_i\}_{i=1}^d\subseteq~\mcx$. An expression of this form is known as a {\em Schmidt decomposition} for $\psi$ and is derived by the singular value  decomposition of $\vecc^{-1}(\psi)$.
The scalars $\{\lambda_i\}_{i=1}^d$  and the integer $d$ are uniquely  defined and  are called the {\em Schmidt coefficients} and the {\em Schmidt rank} of $\psi$, respectively. {{Suppose} ${\psi=\sum_{i=1}^d\lambda_i \, y_i\otimes x_i} $ is a Schmidt decomposition for $\psi$, then {we have that} $\|\psi \|^2_2 =~\sum_{i=1}^d \lambda_i^2$. {Lastly, given $x\in \mathbb{C}^n$ we define ${\rm Diag}{(x):=\sum_{i=1}^n x_i \, e_i e_i^*}$.} 
 
\paragraph{{\bf Quantum mechanics.}} 
\label{sec:qmechanics}
  
In this section, we give a brief overview of the mathematical formulation of  quantum mechanics.  The reader is referred to~\cite{NC00} and \cite{Watrous}  for a more thorough introduction.

According to the axioms of quantum mechanics, associated  to any physical system $\rmx$   is a finite dimensional complex  Euclidean space $\mcx$. The {\em state space} of  $\rmx$ is identified  with the set of unit vectors in $\mcx$. A {\em measurement} on  a   system $\rmx$ is specified  by  a family of Hermitian psd  operators $\{ X_i : i \in \calI\}\subseteq \mch_+(\mcx)$ with the property that $\sum_{i\in \calI}X_i=\mathbb{I}_\mcx$. The set  $\calI$  labels  the set of possible outcomes of the measurement. According to the axioms of quantum mechanics, when the  measurement  $\{ X_i : i \in \calI\}$ is performed on  a system $\rmx$  which is  in state $\psi \in \mcx$ the  outcome $i\in \calI$ occurs  with probability $p(i)=\psi^*X_i \psi.$ Notice that $\{ p(i): i\in \calI\}$ forms  a valid probability distribution since  by the definition of a measurement we have that $p(i)\ge 0$ for all $i\in \calI$ and~$\sum_{i\in \calI}p(i)=1$. 

Consider two quantum systems $\rmx$ and ${\rm Y}$ with corresponding state  spaces $\mcx$ and $\mcy$ respectively. According to the axioms of quantum mechanics the Euclidean  space that corresponds to the {\em joint} system $(\rmx,{\rm Y})$ is given by  the tensor product $\mcx\otimes \mcy$ of the individual spaces. Furthermore, if the systems $\rmx$ and $\rm{Y}$ are independently prepared in states $\psi_1\in \mcx$   and $\psi_2\in \mcy$ then the state of the joint  system is given by   $\psi_1\otimes \psi_2\in \mcx\otimes \mcy$. {A state in $\mcx\otimes \mcy$  of the form  $\psi_1\otimes \psi_2$ for some $\psi_1\in \mcx$, $\psi_2\in \mcy$ is called a {\em product state}. Quantum states that  {cannot be written as  convex  combinations of product states} are    called {\em entangled}.} Lastly, any two measurements  $\{ X_i: i \in \calI\}\subseteq \mch_+(\mcx)$ and $\{ Y_j: j\in \mathcal{J}\}\subseteq \mch_+(\mcy)$ on the individual systems $\rmx$ and  $\rm{Y}$ define a {\em product measurement} on the joint system with outcomes $\{(i,j): i\in \calI, j\in \mathcal{J}\}$. The corresponding measurement operators are given by $\{ X_i \otimes Y_j: i\in \calI, j\in \mathcal{J}\}\subseteq\mch_+(\mcx\otimes \mcy)$ and the probability of  outcome $(i,j)\in \calI\times \mathcal{J}$, {when measuring the quantum state $\psi$}, is equal to~$\psi^*(X_i\otimes Y_j)\psi$. 
 
\paragraph{{\bf Convex analysis and linear conic programming.}} \label{sscn:conic}
 
In this section, we  introduce   conic  programming and state   the duality results that are relevant to this work.   For additional  details, the reader is referred to 
\cite{BTN13}.

Let $\mcv$  be a  finite dimensional vector space  equipped with inner product $\la \cdot, \cdot \ra$. Given a subset $A\subseteq \mcv$ we denote by  ${\rm cl}(A)$ the  {\em closure} of $A$ and by ${\rm int}(A)$ the {\em interior} of $A$ with respect to the topology induced by the inner product.
A subset $\calK\subseteq \mcv$ is called a {\em cone} if   $X \in \calK$ implies that $\lambda X \in \calK$ for all $\lambda \ge 0$. A cone $\calK$ is {\em convex} if $X,Y \in \calK$ implies that $X+Y \in \calK$. For  any cone  $\calK$ we can define its {\em dual cone}, denoted by $\calK^{\ast}$, given by
\[
\calK^{\ast} := \left\{S\in \mcv : \inner{X}{S} \ge 0  \text{ for all } X \in \calK \right\}.
\]
 The dual cone $\calK^\ast$ is always closed. A cone $\calK$ is called {\em self-dual} if $\calK=\calK^*$. For every convex cone~$\calK$ we have that $(\calK^*)^*={\rm cl}(\calK)$.   As a consequence  a cone $\calK$ is closed if and only if $\calK=(\calK^*)^*$.

Consider two  finite dimensional inner-product  
spaces  $\mcv$ and $\mcw$  and a convex cone $\calK\subseteq \mcv$. 
A {\em linear conic program} (over the cone $\calK$) is specified  by  a triple $(C,\calL, B)$ where $C\in \mcv$, $ B\in \mcw$ and $\mathcal{L}: \mcv \rightarrow \mcw $ is a linear transformation. To such a triple we associate two  optimization problems:  
{\begin{equation*} 
\begin{aligned} 
\textup{Primal problem (P)} & \quad p:= \sup \{ \inner{C}{X} \, : \, \calL(X)= B, \; X \in \calK \} \\ 
\textup{Dual problem (D)} & \quad d:= \inf \{ \inner{B}{Y} \, : \, \calL^{\ast}(Y) - C \in \calK^{\ast}, \; Y\in \mcw \},
\end{aligned} 
\end{equation*}} 
\noindent  referred to as the {\em primal} and the {\em dual}, respectively. 
 For brevity, sometimes we drop the ``linear'' and just refer to them as \emph{conic programs}. {We call $p$ the {\em primal value} and $d$ the {\em dual value} of  $(C,\calL, B)$.}

Linear conic programming constitutes  a wide generalization of  several well-studied models of mathematical optimization. For example, setting $\mcv=\R^n$, $\mcw=\R^m$ (equipped with the canonical inner-product) and $\calK=\R_+^n$  then (P) and (D) form   a pair of primal-dual {\em linear programs}. Furthermore, setting  $\mcv=S^n$, $\mcw=\calS^m$ (equipped with the  Hilbert-Schmidt inner product)  and $\calK=\pos^n$ then (P) and (D) form a pair of  primal-dual {\em semidefinite programs}.  
 
{A conic  program $(C,\calL, B)$  is    \emph{primal feasible}} if 
$ {\{X\in \mcv : \calL(X) = B\} \cap \calK\ne \emptyset} $
and   \emph{primal strictly feasible} if $\{X\in \mcv : \calL(X) = B \} \cap \interior(\cal K)\ne \emptyset$.
 Analogously, the conic program $(C,\calL, B)$ is called {\em dual feasible}  if there exists $Y\in \mcw$ such that $\calL^{\ast}(Y) -C\in \calK^{\ast}$
 and 
  {\em dual strictly  feasible} if there exists  $Y\in \mcw$ such that $\calL^{\ast}(Y) -C\in {\rm int}(\calK^{\ast})$. 
  The set of feasible solutions of a linear programming problem is called a {\em polyhedron} and the set of feasible solutions of a semidefinite programming problem is called a {\em spectrahedron}. 

{Conic programs share some of the duality theory available for linear and semidefinite programs. In particular, the dual value is always an upper bound on  the primal value and, moreover, equality and attainment hold assuming appropriate constraint qualifications.}

\begin{theorem}\label{dualitytheorems}
{Let $(C,\calL, B)$ be a linear conic program over  a convex cone $\calK$.}
\begin{itemize}
\item[(i)](Weak duality)  If $X$ (resp. $Y$) is primal (dual) feasible then~${\la C,X\ra\le \la B,Y\ra}.$
\item[(ii)] (Strong duality) Suppose $\calK$ is a closed convex cone. If 
 the  primal is strictly feasible and $p<+\infty$ we have that   $p=d$ and moreover the dual value is attained. Symmetrically, if the  dual program is strictly feasible and $d>-\infty$ then  $p=d$ and the primal value is attained. 
\end{itemize}
\end{theorem}
{Strong duality results in the conic programming setting  are  stated for {\em closed} convex cones. For a closed convex cone $\calK$ we have   $\calK=(\calK^*)^*$  so   the duality results are symmetric with respect to the primal and the dual problem. 
Since the $\CPSD$ cone is not known to be closed we cannot apply Theorem \ref{dualitytheorems} \textit{(ii)} for $\calK=\CPSD$. 
In Section \ref{sect:duality} we apply Theorem \ref{dualitytheorems} \textit{(ii)} to ${\rm cl}(\CPSD)$ and~$\CPSD^{\ast}$}.  

\section{Correlations as projections of affine slices of convex cones}
\label{sec:correlations}

In this section we study the sets of classical, quantum, no-signaling and unrestricted correlations and express them in a uniform manner as projections of affine slices of appropriate convex cones. Using these characterizations we identify a spectrahedral outer approximation  to the set of quantum correlations which is contained  in the first level of the NPA hierarchy and a sufficient condition for showing that the set of quantum correlations is closed. 

\paragraph{{\bf An algebraic characterization of quantum correlations.}} 
\label{sec:algquantum}
 
We start by  investigating  the structure of  the quantum states that can be used to  generate a quantum correlation and show they can {be} taken to have a specific form. {We  make  use  the fact  that if $X\in \mch_+^n$ then $YXY^*\in \mch_+^n,$ for any $n\times n$ matrix $Y$.}
  
\begin{lemma} 
\label{lem:symmetricschmidt} 
Any quantum correlation $\pabst\in \calQ$  can be generated by a  quantum state of the form $\psi = \sum_{i =1}^d \sqrt{\lambda_i} \; e_i\otimes e_i \in \C^d \otimes \C^d$ {(for some $d\ge 1$)}, where $\sum_{i=1}^d \lambda_i=1$ and $\{ e_i : i \in [d] \}$ is the standard basis for $\C^d$ .
\end{lemma}

\begin{proof}  
Since $p\in \calQ$   there exist a  quantum state
$\psi \in \mcx\otimes \mcy$ 
and quantum measurement operators 
$\{X^s_a\}_{a\in A}\subseteq \mch_+(\mcx) $ and $\{Y^t_b\}_{b\in B}\subseteq \mch_+(\mcy)$ satisfying  $ p(a,b|s,t)=\psi^*(X^s_a\otimes Y^t_b)\psi,$   for all $ a \in A, b \in B, s \in S, t \in T. $
By  the Schmidt decomposition,  the  vector   
$\psi \in \mcx \otimes \mcy$ can be expressed  as 
$\psi = \sum_{i=1}^d \sqrt{\lambda_i} \, x_i \otimes y_i,$  
where $\sum_{i=1}^d \lambda_i = 1$   and $\{ x_i \}_{i=1}^d \subseteq \mcx$, $\{ y_i\}_{i=1}^d \subseteq \mcy$ are orthonormal sets of vectors. 
Set   $U := \sum_{i=1}^d \, e_i x_i^*$ and  $U' := \sum_{i=1}^de_i{y_i}^*$ and note that 
\begin{itemize}
\item[$\bullet$] $\tilde{\psi}:=(U \otimes U')\psi=\sum_{i=1}^d \sqrt{\lambda_i} \; e_i\otimes e_i \in \C^d \otimes \C^d$ is a valid quantum state. 
\item[$\bullet$] {For all $s$, the matrices $\{\tX^s_a:=UX^s_aU^*: a\in A\}$ are a 
measurement on~$\C^{d}$}. 
\item[$\bullet$] {For all $t$, the matrices $\{ \tilde{Y}^t_b := U' Y^t_b (U')^*:  b\in B\}$ 
are a measurement on~$\!\C^{d}$.}
\item[$\bullet$] $ \tilde{\psi}^*(\tX^s_a\otimes \tilde{Y}^t_b)\tilde{\psi}=\psi^*(X^s_a\otimes Y^t_b)\psi$,\   for all $a \in A, b \in B, s \in S, t \in T$.
\end{itemize}
Thus, the strategy given by the quantum state  $\tilde{\psi}$ and the quantum measurements $\{ \tX^s_a\}_{a\in A}$ and $\{ \tilde{Y}^t_b\}_{ b\in B}$ also generates  $p$ and has the desired properties. 
\end{proof}

Based on  Lemma \ref{lem:symmetricschmidt} we arrive at  a  new algebraic  characterization of the set  of quantum correlations that is of central importance in Section \ref{sec:conicformulationscorrelations}. 

\begin{theorem}\label{thm:qcorrelationvector} {For any  $p \in \R^{|A\times B \times S \times T|}$, the following are equivalent:} 
\begin{itemize}
 \item[(i)]  $p$ is a quantum correlation. 
 \item[(ii)] There exist operators $K, \{X^s_a\}_{s,a}, \{Y^t_b\}_{t,b} \in \mch_+^d$ (for some  $d \geq 1$)~{such that} 
\begin{equation} 
\label{eq:qcorrelationvector1} 
\begin{aligned}
&   \inner{K}{K} = 1,\\
& \sum_{a} X^s_a = \sum_{b} Y^t_b = K, \ \forall s,t,\\
& p(a,b|s,t)=\la X^s_a, Y^t_b\ra, \ \forall a,b,s,t.
\end{aligned}
\end{equation} 
\end{itemize}
\end{theorem}
 
\begin{proof}  Let $\pabst\in \calQ$. By \eqref{eq:quantumcorrelations} there exist a quantum state $\psi \in \mcx\otimes \mcy$ 
and measurements  
$\{\tX^s_a\}_{a\in A}\subseteq \mch_+(\mcx) $ and $\{\tY^t_b\}_{b\in B}\subseteq \mch_+(\mcy)$~with 
\be
p(a,b|s,t)=\psi^*(\tX^s_a\otimes \tY^t_b)\psi,  \text{ for all } a,b,s,t.
\ee
By Lemma \ref{lem:symmetricschmidt} we may assume $\mcx = \mcy = \C^d$, for some integer $d \geq 1$ and that 
$\psi =\sum_{i=1}^d \sqrt{\lambda_i} \; e_i\otimes e_i,$
  where $\sum_{i=1}^d \lambda_i=1$. 
    {Define 
${K :=\sum_{i=1}^d \sqrt{\lambda_i} \; e_i e_i^* \in {\mch_+^d}}$ and notice  that $\vecc(K)=\psi$ and $\inner{K}{K} = 1$. {Set}  
 $X^s_a := K^{1/2} (\tX^s_a) K^{1/2}$  for all $a$, $s$ and $Y^t_b := K^{1/2}(\tY^t_b)^\sfT K^{1/2}$, for all $b$, $t$. 
 These operators satisfy $\sum_{a\in A} X^s_a =\sum_{b\in B} Y^t_b = K$, for all $s\in S, t\in T$. }

{Using the definitions above and  properties of the $\vecc$ map (cf. \eqref{vecprop})~we~have} 
{\begin{equation} \label{eqn:calc}
\inner{X^s_a}{Y^t_b} 
=  
\vecc(K)^*(\tX^s_a\otimes \tY^t_b) \vecc(K) 
= 
\psi^*(\tX^s_a \otimes \tY^t_b)\psi=p(a,b|s,t),
\end{equation}}
 for all $a\in A, b\in B, s\in S, t\in T$ 
 and thus \eqref{eq:qcorrelationvector1} is feasible.   

{Conversely  let  $K, \xsa, \ytb $  be   feasible for \eqref{eq:qcorrelationvector1}. Without loss of generality, we may assume $K$ has full rank. 
Define $\psi := \vecc(K)$ and notice  that  $\norm{\psi}_2 = 1$. {For all $a,s$ set 
 $\tX^s_a := K^{-1/2} X^s_a K^{-1/2}$  and for all $b,t$ set    $\tY^t_b := \left(K^{-1/2} Y^t_b K^{-1/2}\right)^\sfT$. Since $K^{-1/2}\in \mch^d_+$ we have that $\tX^s_a,\tY^t_b\in \mch^d_+$, for all $a,b,s,t$} ({where we also use the fact  that $X\in \mch_+^n$ if and only if $X^\sfT\in \mch_+^n$}).  Clearly, these operators satisfy  $\dsum_{a \in A} \tX^s_a = \mathbb{I}_d= 
\dsum_{b \in B} \tY^t_b =\mathbb{I}_d,$ for all $s\in S, t\in T$. }
Reversing the calculation in (\ref{eqn:calc}) we get that $p(a,b|s,t)=\inner{X^s_a}{Y^t_b}=\psi^*(\tX^s_a \otimes \tY^t_b) \psi,$
for all $a,b,s,t$ which shows that $p\in \calQ$. 
\end{proof}
 \begin{remark}
A close inspection of the proof of Theorem \ref{thm:qcorrelationvector}  allows us to  explicitly work out the dependency of the parameter  $d$ on the dimension of the underlying quantum system. Specifically, for a correlation $p \in \calQ$ that is generated by a  state  $\psi\in \mcx\otimes \mcy$ there exist Hermitian psd matrices of size $d\le \min \{ \dim(\mcx), \dim(\mcy) \} $ that satisfy  \eqref{eq:qcorrelationvector1}. Conversely, if \eqref{eq:qcorrelationvector1} has a feasible solution with matrices  (real or complex) of size $d\ge 1$  then the correlation $\pabst$ can be generated by a state in $\C^d\otimes \C^d$. This observation can be used to derive a lower bound on the dimension of a Hilbert space needed to generate an arbitrary quantum correlation~\cite{SVW15}. 
\end{remark}

\begin{remark} As another by-product of Theorem \ref{thm:qcorrelationvector} 
{we characterize}
 the quantum correlations that can be generated using  a maximally entangled state. Specifically, it follows easily from the proof of Theorem~\ref{thm:qcorrelationvector}  that a quantum correlation $p\in \calQ$ can be generated using the {\em $d$-dimensional maximally entangled state} {$\psi_d:=\vecc\big({1\over \sqrt{d}} \mathbb{I}_d\big)$}  if and only if \eqref{eq:qcorrelationvector1} is feasible with $K={1\over \sqrt{d}}\mathbb{I}_d$. 
\end{remark}

\paragraph{{\bf An algebraic characterization of classical  correlations.}}
\label{sec:algclassical}
{As every classical correlation is also quantum,  
any classical correlation} admits a representation as a quantum correlation for some appropriate  choice of quantum state and measurement operators. In the next result we show that a correlation  is classical   if and only if    \eqref{eq:qcorrelationvector1} admits a solution with diagonal psd matrices.

\begin{theorem}\label{thm:ccorrelationvector} {For any  $p \in \R^{|A\times B \times S \times T|}$, the following are equivalent:}  
\begin{itemize}
 \item[(i)] $p$ is a classical correlation. 
 \item[(ii)] $p$ can be generated by {{\em diagonal} measurement operators} and a state of the form 
$\psi = \sum_{i =1}^n \sqrt{\lambda}_i \; e_i\otimes e_i$, where $\sum_{i=1}^n \lambda_i=1, \ \lambda_i\ge 0 \text{ for } i\in [n]$.  
  \item[(iii)] There exists  a solution to \eqref{eq:qcorrelationvector1} with diagonal matrices. 
\end{itemize}
\end{theorem}

\begin{proof} 
\textit{(i)} $\Longrightarrow$ \textit{(ii)}:   By definition, for any classical correlation $p\in\calP$ 
there exist nonnegative scalars $ k_i\ge 0, \ x_{a}^{s,i}\ge 0$, and $ y_{b}^{t,i}\ge 0$ satisfying $ p(a,b|s,t)=\sum_{i=1}^n k_ix_a^{s,i}y_b^{t,i}, \text{ for all } a\in A,b\in B,s\in S,t\in T, $
where 
 $ \sum_{i =1}^n k_i = 1$, 
  $\sum_{a \in A} x_{a}^{s,i} =1$ for all $i\in[n], s\in S$ and $ \sum_{b \in B} y_{b}^{t,i} = 1$ for all $i\in [n], t\in T$.~Set 
\begin{itemize}
\item[$\bullet$] $\psi := \sum_{i=1}^n \sqrt{k_i}\ e_i\otimes e_i$, which is a quantum state of the form in Lemma~\ref{lem:symmetricschmidt},
\item[$\bullet$] $X^s_a := \sum_{i=1}^n x^{s,i}_ae_ie_i^*$, for $s \in S$, $a \in A$, which are diagonal, 
\item[$\bullet$] $Y^t_b := \sum_{i=1}^ny^{t,i}_be_ie_i^*$, for $t \in T$, $b \in B$, which are diagonal. 
\end{itemize} 
A straightforward  calculation shows that the state  $\psi$ and the measurements $\{X^s_a\}_{s\in S,a\in A}$ and $ \{Y^t_b\}_{t\in T,b\in B}$ generate $p$.  
 
\textit{(ii)} $\Longrightarrow$ \textit{(i)}: {Suppose $p$ can be generated by $\psi = \sum_{i=1}^n \sqrt{\lambda_i} \, e_i \otimes e_i$ and diagonal measurement   operators $\{X^s_a\}_{s, a}$ and $ \{Y^t_b\}_{t,b}$. Set} 
\begin{itemize} 
\item[$\bullet$] $k_i := \lambda_i,\ $ for all $i \in [n]$, 
\item[$\bullet$] $x^{s,i}_a := X^s_a [i,i], \ $ for all $i\in [n], s\in S$ and $a \in A$, 
\item[$\bullet$] $y^{t,i}_b := Y^t_b[i,i], \ $ for all $i\in [n], t\in T$ and $b \in B$,
\end{itemize} 
and notice that this defines a classical strategy that generates the correlation~$p$. 

\textit{(ii)} $\Longleftrightarrow$ \textit{(iii)}: This is clear from the proof of Theorem~\ref{thm:qcorrelationvector} noting that $K$ being diagonal and satisfying $\inner{K}{K} = 1$, implies that the quantum state $\psi := \vecc(K)$ is of the required form.
\end{proof} 

\paragraph{{\bf Conic characterization of correlations.}} 
\label{sec:conicformulationscorrelations}
 Recall that for a convex cone  ${\calK \subseteq \NN}$,  we say that $\pabst$ is a $\calK$-{\em correlation}, denoted by $p\in \corr(\calK),$   if and only if there  
exists $X\in \calK^N$ such that 
  $\inner{J_{i,j}}{X}  = 1,$ for all $  i,j \in S\cup T, $ and  $\xasbt = p(a,b|s,t),$  for all $  a,b,s,t$. 
(cf.  Definition~\ref{def:kcorr}). 
In this section we show that the sets of  classical, quantum, no-signaling and unrestricted correlations can be expressed  as the sets of conic correlations for appropriate   convex cones. The characterizations for the quantum and classical case rely on   
Theorem~\ref{thm:qcorrelationvector} and  Theorem \ref{thm:ccorrelationvector},  respectively. 
 
Recall that  for  $i,j\in S\cup T$ (resp. $i,j\in \{0\}\cup S\cup T$) we set  $J_{i,j}$ to be the matrix which  acts on a matrix $X\in \calS^N$ (resp. $X\in \calS^{N+1})$  by summing all 
 {entries in block $X_{i,j}$ (cf. \eqref{eq:jmatrix} and \eqref{eq:jmatrix2})}.
We start  with  a  geometric lemma of central importance in this section.  For the definition of a Gram matrix of a family of vectors  recall \eqref{eqn:gram}.
 
\begin{lemma}\label{lem:vectors} 
{Consider  
vectors $\sxsa$ and $\sytb$ in some 
Euclidean space~$\mcx$.} 
\bi 
\item[(a)] For   $X := \gram(\sxsa, \sytb)$ 
 the  following are equivalent:  
\begin{itemize}
\item[(i)] {$\exists k\in \mcx$ such that  $\la k,k\ra=1,$ and \
${\sum_{a} x^s_a=\sum_{b} y^t_b=k}$, for all $s, t$.}  
\item[(ii)] $ \la J_{i,j}, X\ra=1,\ $ for all $i,j\in S\cup T$. 
\end{itemize}
\item[(b)]  Set  $\tX:=\gram(k,\sxsa,\sytb)$ where $k\in \mcx$ with $\la k,k\ra=1$. The following are equivalent:
\bi
\item[(i)] $\sum_{a\in A} x^s_a=\sum_{b\in B} y^t_b=k,\ $ for all $s,t$.
\item[(ii)] $ \la J_{i,j}, \tX\ra=1,\ $ for all $i,j\in \{0\}\cup S\cup T$.  
\ei
\ei
\end{lemma} 

\begin{proof} We start with part $\textit{(a)}$. 
To show $\textit{(i)}$ implies $\textit{(ii)}$, consider $i \in S$ and $j \in T$ and notice that
\[ 
\inner{J_{i,j}}{X} 
= 
\sum_{a \in A} \sum_{b \in B} X[(i,a),(j,b)]
= 
\sum_{a \in A} \sum_{b \in B} \inner{x_a^i}{y_b^j} 
=  
\inner{k}{k} 
= 1.
\]

For the other direction, define $x_i := \sum_{a\in A} x^i_a $ for all $i\in S$  and  $y_i := \sum_{b\in B} y^i_b$     for all  $i\in T$. Notice that for any $i,j\in S$ the equation   $\la J_{i,j},X\ra=1$ is equivalent to $\la x_i, x_j\ra=1$. This implies that $\la x_i-x_{j},x_i-x_j\ra=0$  for all $i, j \in S$ and thus $x_i=x_j$ for all $i, j \in S$. Similarly we have that $y_i=y_j$ for all $i, j \in T$. Lastly, fix any  $i \in S$ and $j\in T$ and notice that $\la J_{i,j},X\ra=1$ implies that $\la x_i,y_j\ra=1$. As before this shows that $x_i = y_j$. 

We proceed with part \textit{(b)}. It is easy to see that \textit{(i)} implies \textit{(ii)}. For the other direction we have from part \textit{(a)} that there exists $k'\in \mcx$ such that $\sum_{a\in A} x^s_a=\sum_{b\in B} y^t_b=k',\ $ for all $s\in S$, $t \in T$, and $\la k',k'\ra=1$. It suffices to show that $k=k'$. For this, notice that
$\la k-k',k-k'\ra=2-2\la k,k'\ra=2-2\sum_{a\in A}\tX[0,(s,a)]=0,$ and the proof is concluded.
\end{proof}
 
We now  state and prove our  main result in this section. 

\begin{theorem}\label{thm:correlations}
For any Bell scenario we have that 
$$\calC=\corr(\CP),\  \calQ=\corr(\CPSD),\  \NS=\corr(\nso),  \  \calP=\corr(\NN).$$ 
\end{theorem}
  
\begin{proof}We first show   $\calQ=\corr(\CPSD)$.
{By Theorem \ref{thm:qcorrelationvector} we have that $\pabst\in \calQ $ if and only if  there exist operators $\{X^s_a\}_{s,a},\{Y^t_b\}_{t,b},K$ satisfying \eqref{eq:qcorrelationvector1}. Setting $X := \gram(\{X^s_a\}_{s,a},\{Y^t_b\}_{t,b},K)$ it follows by Lemma \ref{lem:vectors} $(a)$ that $\la J_{i,j},X\ra=1,$ for all $i,j\in S\cup T$. As $X\in \CPSD$, this gives~$p\in \corr(\CPSD)$.}
  
{Conversely, fix $p\in \corr(\CPSD)$ and let   $X=\gram(\{X^s_a\}_{s,a},\{Y^t_b\}_{t,b})\in \CPSD^N$ 
satisfying \eqref{kcorrelations}. 
 By Lemma \ref{lem:vectors} $(a)$ there exists  a Hermitian psd matrix $K$ with   $\la K,K\ra =1$ and   $\sum_{a} X^s_a=\sum_{b} Y^t_b=K,\  \forall s,t$. 
 By Theorem \ref{thm:qcorrelationvector} we get~$p\in \calQ$.}
 
 The case $\calC=\corr(\CP)$ follows  similarly by Theorem \ref{thm:ccorrelationvector}. Lastly, $\NS=\corr(\nso)$ and $\calP=\corr(\NN) $
  follow  from the definitions of $\nso$ and $\NN$.
\end{proof}  
 
As exemplified  by  Theorem \ref{thm:correlations}  the notion of $\calK$-correlations has significant expressive power as it captures many  correlation sets of physical significance. In our next result we continue the  study of conic correlations and identify a sufficient condition in terms of  the cone $\calK$ so that the corresponding set of correlations $\corr(\calK)$ satisfies the  no-signaling conditions. 

\begin{theorem}\label{thm:corrktoNS}
For any  convex cone $\calK\subseteq \DNN$ we have that $\corr(\calK)\subseteq \NS$.
\end{theorem}
\begin{proof}{For any   $\pabst\in\corr(\calK)$  there exists   $X\in\DNN^N $ such that   $\la J_{i,j}, X\ra=1$ for all $i,j\in S \cup T$ and $p(a,b|s,t)=\xasbt$ for all  $a,b,s,t$}. If  $X=\gram(\sxsa,\sytb)$ it follows from Lemma~\ref{lem:vectors} $(a)$  that there exists a vector $k$ such that  
\be\label{cdvefefe}
 \sum_{a\in A} x_a^s = \sum_{b\in B} y_b^t = k,  \ \text{ for all } s\in S, t\in T, \ \text{ and } \  \la k,k\ra=1. 
 \ee
For $s\ne s'\in S$ and $t\in T$ we get from  \eqref{cdvefefe}   that  
\begin{equation*}
\sum_{a\in A} \xasbt 
= 
\sum_{a\in A} \inner{x^s_a}{y^t_b} 
= 
\inner{k}{y^t_b}
= 
\sum_{a\in A} \inner{x^{s'}_a}{y^t_b} 
= 
\sum_{a\in A} X[(s',a),(t,b)],
\end{equation*}
and thus $\sum_{a\in A}p(a,b|s,t)=\sum_{a\in A}p(a,b|s',t)$ for all $t\in T$ and $s\ne s'\in S$. Symmetrically, we have that  $\sum_{b\in B} p(a,b|s,t)=
\sum_{b\in B} p(a,b|s,t')$ 
for all $s\in S$ and $t\ne t'\in T$ and thus $p\in~\NS$.
\end{proof}

Another consequence of  Theorem \ref{thm:correlations}  is a sufficient condition for showing that 
$\calQ$ 
is closed.  Note that  for every cone $\calK\subseteq \NN$, the set of matrices satisfying $ \inner{J_{i,j}}{X}  = 1,$ for all  $i,j \in S\cup T$ is bounded. Consequently, if 
$\calK$ is closed it follows that $\corr(\calK)$ is compact.  
{This implies the following proposition.} 
\begin{proposition}
\label{cor:closedness}If the cone $\CPSD$ is closed then $\calQ$ is also closed. 
\end{proposition}

We conclude this section with a {second 
 formulation} for the sets of $\CP, \CPSD$ and $\DNN$-correlations.  {We use these formulations  
  to
  compare  $\corr(\DNN)$    with the first level of the NPA hierarchy} 
and in 
{Section \ref{sect:corr} where we  recover the conic programming formulations for certain quantum graph parameters. 
 
\begin{lemma}
\label{lem:dnnextrarow}
 {Consider a correlation}    $\pabst\in \calP$ and {define the matrix} 
 \[ {P:=\sum_{a,b,s,t} p(a,b|s,t)\ e_se_t^\sfT\otimes e_ae_b^\sfT}. \] 
 For   any cone $\calK\in \{\CP,\CPSD,\DNN\}$ we have that   $\pabst \in \corr(\calK)$ if and only if there exists a matrix 
\be\label{eq:whatever}
\tilde{X}=\begin{pmatrix}
 1 & x^\sfT & y^\sfT \\
 x & X & P\\
 y & P^\sfT& Y
 \end{pmatrix}\in \calK^{1+N}, 
 \ee
  such that
$\la J_{i,j},\tilde{X}\ra=1, \text{ for all } i,j\in \{0\}\cup S\cup T.$ 
\end{lemma}

\begin{proof} This follows  from the definition of $\corr(\calK)$  combined with  Lemma~\ref{lem:vectors}~\textit{(b)}.
\end{proof}
  
\paragraph{{\bf A spectrahedral outer approximation {for}    quantum correlations.}}
\label{sec:specouterapprox}

In this section we use 
Theorem \ref{thm:correlations} to derive  a new  spectrahedral outer approximation {for}  the set of quantum correlations. Furthermore, we show that our approximation is at least as strong as  the first level of the NPA hierarchy. 

In Theorem \ref{thm:correlations} we showed  that  $\calQ=\corr(\CPSD).$ As  $\CPSD \subseteq \DNN $ we immediately get    a necessary and efficiently verifiable  condition for  membership in the set of quantum~correlations.

\begin{proposition}
\label{cor:dnnupperbound}
For any Bell scenario we have  
$\calQ\subseteq \corr(\DNN)\subseteq \NS.
$
\end{proposition}
 
As already mentioned the set of quantum correlations is a non-polyhedral set whose structure is poorly understood.   In \cite{NPA08} Navascu\'es, Pironio and Ac\'in  constructed a hierarchy of spectrahedral  outer approximations to the set of quantum correlations. 
{The 
 mathematical derivation}  of the NPA  hierarchy is  involved and is beyond the scope of this paper. For the precise definition and its properties the reader is referred to \cite{NPA08}. {In this work we only consider  the first level of the NPA hierarchy,  denoted by~${\rm NPA}^{(1)}, $ that we introduce below.}
 
 For this we need the following  definition. 
 For ${\pabst\in \NS}$ we denote by $p_A(a|s)$ Alice's local marginal probabilities for all $a\in A, s\in S$ and by $p_B(b|t)$ Bob's local marginal probabilities for all $b\in B, t\in T$. Note that these are well-defined by the no-signaling conditions  \eqref{ns1} and \eqref{ns2}.

It is useful to arrange the marginal probabilities  in a vector as follows:  
{\be\label{csdnierifnerg}
p_A(s):=\sum_{a\in A} p_A(a|s)e_a\in \R_+^{|A|},  \ \text{ and }  \ p_B(t):=\sum_{b\in B} p_B(b|t)e_b\in \R_+^{|B|}, 
\ee}
{for all $s\in S$ and $t\in T$, respectively, and} 
\be\label{eq:marginals}
p_A:=\sum_{s,a} e_s\otimes p_A(s)
  \in \R_+^{|S \times A|}  
  \text{ and } \ 
 p_B:=\sum_{t,b} e_t\otimes p_B(t) 
 \in \R_+^{|T\times B|}.
\ee
 
\begin{remark}\label{marginals}
{Let $\calK\in \{\CP,\CPSD,\DNN\}$. {Given $\pabst\in \corr(\calK)$} it follows from the proof of Lemma \ref{lem:dnnextrarow}  that   every feasible solution to  \eqref{eq:whatever} satisfies  $\tilde{X}[0, (s,a)]=p_A(a|s), \ $ for all $a\in A,s\in S$ and $\tilde{X}[0,(t,b)]=p_B(b|t), \ $ for all $t\in T,b\in B$. We make use of this fact  in Theorem~\ref{thm:corrnpa}.}   
\end{remark} 

Using the vectors given in \eqref{eq:marginals} we can now give the description of ${\rm NPA}^{(1)}$.
\begin{definition}
\label{def:NPA}
{Consider a correlation   $\pabst\in \NS$ and} {define} \[ P:=\sum_{a,b,s,t} p(a,b|s,t)\ e_se_t^\sfT\otimes e_ae_b^\sfT, \] 
and $p_A,p_B$ as defined  in \eqref{eq:marginals}.  
Then  $p\in {\rm NPA}^{(1)}$ if and only if there exists a matrix    
\be\label{eq:NPAl}
\tilde{X}:=\begin{pmatrix}
 1 & p_A^\sfT & p_B^\sfT \\
 p_A & X & P\\
 p_B & P^\sfT& Y
 \end{pmatrix}\in \pos^{1+N},\ {satisfying:} 
 \ee
 \bi  
 \item[$(i)$] $  X[(s,a),(s,a')]=\delta_{a,a'}\ p_A(a|s),  \text{ for all } s\in S, a,a'\in A,$
\item[$(ii)$] $  Y[(t,b),(t,b')]=\delta_{b,b'}\ p_B(b|t), \text{ for all } t\in T, b,b'\in B. $
\ei
\end{definition} 
{\begin{remark}\label{NPAalternative}
  Using Lemma \ref{lem:vectors} it is easy to verify  that ${\rm NPA}^{(1)}$ can be expressed as  the projection  (onto the blocks that are indexed by  $S\times T$) of the set of matrices in $\pos^{1+\asbt}$ satisfying the following constraints  
\begin{enumerate} 
\item[$(i)$]    $\la J_{i,j}, \tilde{X}\ra  =1$, for all $i, j\in \{0\}\cup S\cup T$,  
\item[$(ii)$] $\tilde{X}[(s,a),(t,b)] \ge 0$, for all $s\in S$, $t\in T$, $a\in A$, $b\in B$, 
\item[$(iii)$] $\tilde{X}[(s,a),(s,a')] =0$, for all $s\in S$, $a\ne a'\in A$,  
\item[$(iv)$] $\tilde{X}[(t,b),(t,b')] =0$, for all $t\in T$, $b\ne b'\in B$.  
\end{enumerate} 
We make use of this fact in Section \ref{sec:dnn}.  
 \end{remark}}
{Our last result in this section is that the set of $\DNN$-correlations is contained in  ${\rm NPA}^{(1)}$. We start with a simple lemma  that we use in the proof.}  

\begin{lemma}
\label{lem:psdlemma}
{Consider 
$x,y \in \R_+^n$ with $\la e,x\ra=\la e,y\ra=1.$ If  the matrix 
$\begin{pmatrix} 1 & x^\sfT \\
x & {\rm Diag}(y)
\end{pmatrix}$ 
is positive semidefinite then we have that  $x=y$.}  
\end{lemma}

\begin{proof}{By  Schur complements (e.g. see \cite{BTN13}) we have   $\begin{pmatrix} 1 & x^\sfT \\
x & {\rm Diag}(y)
\end{pmatrix}\in \pos^{n+1}$ if and only if ${\rm Diag}(y)-xx^\sfT \in \pos^n$. Note that 
$\la ee^\sfT, {\rm Diag}(y)-xx^\sfT\ra 
=0$. Since ${\rm Diag}(y)-xx^\sfT$ is psd   we get    $({\rm Diag}(y)-xx^\sfT)e=0$ (where we  use the well-known fact that for $X\in \calS^n_+$ we have $x^\sfT Xx=0$ if and only if $Xx=0$). Lastly, as  $\la e,x\ra=~1$, it follows from  the preceding  equality that $x=y$.}
\end{proof}
We are now ready to  prove the last result in this section. 
 
\begin{theorem}
\label{thm:corrnpa}
For any Bell scenario we have that 
$\corr(\DNN)\subseteq {\rm NPA}^{(1)}.$
\end{theorem}
\begin{proof}Consider a correlation $p\in \corr(\DNN)$. By Theorem \ref{thm:corrktoNS} we have that $p\in \NS$ and thus the marginal probability distributions  $p_A$ and $p_B$ are well-defined. By  Remark \ref{marginals}   there exists
\be\label{eq:extrarowcol1}
\tilde{X}:=\begin{pmatrix}
 1 & p_A^\sfT & p_B^\sfT \\
 p_A & X & P\\
 p_B & P^\sfT& Y
 \end{pmatrix}\in \DNN^{1+N},
 \ee
 satisfying $\la J_{i,j},\tilde{X}\ra=1, \text{ for all } i,j\in \{0\}\cup S\cup T.$
Fix $s\in S$ and $a\ne a'\in A$ and set   $E^s_{a,a'}\in {\pos^{1+N}}$ to be the matrix with entries {$E^s_{a,a'}[(s,a),(s,a)] =1$, $E^s_{a,a'}[(s,a'),(s,a')]=1$,  $E^s_{a,a'}[(s,a),(s,a')]= -1$, $E^s_{a,a'}[(s,a'),(s,a)]=-1$ and $0$ otherwise}. Furthermore, define  
\be\label{eq:newz}
X':=\tilde{X}+\tilde{X}[(s,a),(s,a')]E^s_{a,a'},
\ee
and notice that $X'[(s,a),(s,a')]=0$. Moreover, since     $\tilde{X}\in \DNN^{1+N}$ 
we have   that $X'\in \pos^{1+N}$ and since $\la J_{i,j},E^s_{a,a'}\ra=0$ 
 it follows from  \eqref{eq:newz} that  $\la J_{i,j},X'\ra=1$ for all $i,j\in \{0\}\cup S\cup T$.
Clearly, this argument can be repeated  for all $s\in S, a\ne a'\in A$ and symmetrically  for all $t\in T$ and $ b\ne b'\in B$. In this way we construct a matrix 
\be
Z:=\begin{pmatrix}
 1 & p_A^\sfT & p_B^\sfT \\
 p_A & Z_1 & P\\
 p_B & P^\sfT& Z_2
 \end{pmatrix}\in \pos^{1+N},
 \ee
  satisfying $\la J_{i,j}, Z\ra=1$ for all $i,j\in \{0\}\cup S\cup T, $ 
  $Z_1[(s,a)(s,a')]=0$ for {every} $s\in S, a\ne a'\in A$, and $Z_2[(t,b),(t,b')]=0$ for {every} $t\in T, b\ne b'\in B$. 
It remains to show that $Z[(s,a),(s,a)]=p_A(a|s)$ for all $a\in A, s\in S$   and ${Z}[(t,b),(t,b)]=p_B(b|t)$ for all  $t\in T$ and $b\in B$. For  this,  {fix $ s\in S$} and notice that the principal submatrix of $Z$ indexed by $\{[0,0]\}\cup \{[0,(s,a)]: a\in A\}$ is given by  $\begin{pmatrix} 1 & p_A(s)^\sfT \\p_A(s) & {\rm Diag}(y)\end{pmatrix}\in \pos^{1+|A|}$. Since  $\la y,e\ra=1$ it follows by   Lemma~\ref{lem:psdlemma} that $y=p_A(s)$. {Since the same argument can be repeated for all other diagonal blocks of $Z$,  the proof is concluded.}
\end{proof}

{We do not know if  the containment given in Theorem~\ref{thm:corrnpa} is   strict.
{One 
difficulty in proving the converse inclusion is that for any matrix feasible for \eqref{eq:NPAl}  we do not have control of the signs of the entries in the off-diagonal~blocks.}}
   
\section{Conic programming formulations for  game values}
\label{sec:conicprogrammingformulations}
 
In this section  we study the value of a nonlocal game when   the players  use strategies that generate classical, quantum, no-signaling or  unrestricted correlations. By  Theorem \ref{thm:correlations}, the classical, quantum, no-signaling and unrestricted values 
can be formulated as linear conic programs over appropriate convex cones. Any conic program has an associated dual  which we  derive in our setting   and investigate its properties. This allows us to identify a sufficient condition for showing that  the $\CPSD$ cone is {\em not} closed. Furthermore, we  identify a new SDP upper bound to the quantum value of an arbitrary nonlocal game which we show  is at most the value of the SDP obtained  when we optimize over the first level of the NPA hierarchy.  Lastly, we show that the problem of deciding whether a nonlocal game admits a perfect  $\calK$-strategy is equivalent to deciding the feasibility of a linear conic program over $\calK$.  In particular,  deciding whether a nonlocal game admits a perfect quantum strategy is equivalent to the feasibility of a linear conic program over the $\CPSD$ cone.  

\paragraph{{\bf Primal formulations.}}

Recall that  the maximum probability of winning a game $\calG$   {using $\calK$-correlations~is} given by 
\be\label{generalconic}\tag{$\mathcal{P}_{\calK}$} 
   \omega(\calK, \calG) := \text{sup}\{ \inner{\hat{C}}{X}:  \la J_{i,j},X\ra=1, \text{ for all }\  i,j\in S\cup T,
\ X \in  \calK\},
\ee
where $\hat{C}$ is the symmetric cost matrix defined in \eqref{hatc} {and the matrices $J_{i,j}$ are defined in~(\ref{eq:jmatrix}).}  
 
By Theorem \ref{thm:correlations} the sets of classical, quantum, no-signaling and unrestricted correlations can be expressed as the set of $\calK$-correlations over some appropriate convex cone $\calK\subseteq \NN^N$.  Specifically, we have:
  
\begin{theorem}
\label{cor:gamevalues}
For any nonlocal game $\calG(\pi, V)$  we have: 
\begin{itemize}
\item[(i)]  The classical value $\cvalue$ equal to $\omega(\CP, \calG)$. 
\item[(ii)] The quantum value $\qvalue$ equal to $ \omega(\CPSD, \calG)$. 
\item[(iii)] The no-signaling value $\nsvalue$ equal to $  \omega(\nso, \calG)$.  
\item[(iv)] The unrestricted value $\unresvalue $ equal to $\omega(\NN,\calG)$. 
\end{itemize} 
\end{theorem}
  
Note that  \eqref{kconic} is a linear conic program over the convex cone $\calK$. 
Our next goal is to  apply the  theory of linear conic optimization to \eqref{kconic}  to   understand how the various values relate to each other and to study their properties.      
 
\paragraph{{\bf Dual formulations.}} 
\label{sect:duality}   
The dual conic program associated to  \eqref{generalconic} is given~by:  
\be\label{conicdual}\tag{$\mathcal{D}_\calK$} 
{\xikvalue :=	\text{inf} \bigg\{ \sum_{i,j \in S \cup T} v_{i,j} : 
 \sum_{i,j \in S \cup T} v_{i,j} J_{i,j} - \hat{C} \in \calK^* \bigg\}.}  
\ee 
  
We start  by  analyzing 
the  primal-dual pair of conic programs  ($\mathcal{P}_\calK$) and ($\mathcal{D}_\calK$).  
By  weak duality (cf. Theorem \ref{dualitytheorems} $\textit{(i)}$)  the optimal value of the dual program upper bounds the optimal value of the primal, i.e., for  any  game $\calG$ we have $\kvalue \leq \xikvalue $. For this to hold with equality, it suffices to determine whether  strong duality holds for the primal or the dual (cf. Theorem \ref{dualitytheorems} \textit{(ii)}). 
Note that for  any cone $\calK\subseteq \pos^N $   the primal program  $(\mathcal{P}_\calK)$ is not    strictly feasible. To see this,  
fix  indices $i \in S$, $j \in T$, and define  the (nonzero) psd~matrix 
\[ {M := J_{i,i} + J_{j,j}- 2J_{i,j} \in \pos^N}.\]  
Any matrix $X$ feasible for  $(\mathcal{P}_\calK)$ satisfies  $\inner{M}{X} = 0$ and so  $X\not \in \intr(\calK)\subseteq~\mathcal{S}^N_{++}. $

Also, notice that if the cone $\calK$ is not closed then  we cannot apply strong duality directly  to the primal-dual pair. However, 
as  we now show,  under the additional assumption that $\calK$ is a closed convex cone,  strong duality holds for the primal-dual pair of conic programs ($\mathcal{P}_\calK$) and~($\mathcal{D}_\calK$). 
 
\begin{proposition}\label{thm:strongduality}
Consider  a  game $\calG$ and let $\calK \subseteq \NN$ be a  closed convex cone such that $(\mathcal{P}_\calK)$ is primal feasible. Then we have that   $\kvalue= \xikvalue$ and moreover there exists an optimal solution for~$(\mathcal{P}_\calK)$.  
\end{proposition}

\begin{proof} 
Since $(\mathcal{P}_\calK)$ is feasible, the  dual value $\xikvalue$  is bounded below by $0$. It remains  to show that the dual program is strictly feasible for the range of cones we consider.  
Notice that the  program \eqref{conicdual} is strictly feasible for $\calK = \NN$ since $ \intr(\NN)= \set{ X : X[i,j] > 0 \textup{ for all } i,j}$ 
and $\sum_{i,j \in S \cup T} v_{i,j} J_{i,j} - \hat{C} \in \intr(\NN)$
by setting each $v_{i,j}$ to be a very large positive constant. 
Furthermore, for  $\calK \subseteq \NN$ we have  that
$\NN = \NN^* \subseteq \calK^*$ implying $\intr(\NN) \subseteq \intr(\calK^*)$. Thus \eqref{conicdual} is  strictly feasible for all cones  $\calK \subseteq \NN$. The proof is concluded  by  Theorem \ref{dualitytheorems}~\textit{(ii)}.
\end{proof}

{Since $\CPSD^*=({\rm cl}(\CPSD))^*$, we get the following corollary.}  

\begin{corollary}\label{quantumvalueclosure}
For any nonlocal game $\calG$ we have 
$\pvalueclcpsd = \dvaluecpsd.
$
\end{corollary}
 
Recall that the $\CPSD$ cone is not known to be closed \cite{LP14,BLP15}.  It follows from  Corollary~\ref{quantumvalueclosure} that   a sufficient condition for showing that the  cone $\CPSD$ is {\em not} closed is to identify  a game $\calG$ 
for which $\pvaluecpsd<\dvaluecpsd$.
 
\paragraph{{\bf The Feige-Lov\'asz SDP relaxation.}}
\label{sec:dnn}  

Notice that the tractability of the conic program \eqref{generalconic} depends on the underlying cone $\calK$. In this section we focus on the case $\calK=\DNN$ for which   \eqref{generalconic} becomes  an instance of a semidefinite  program. Specifically, using the definition of $\corr(\DNN)$ we have:
\begin{equation*}
\omega(\DNN, \calG)= \text{max} \{  \la \hat{C}, X\ra :   \la J_{i,j}, X\ra  =1,  \forall  i,j\in S\cup T, \ 
    X\in  \dnn^{\asbt}\},
 \end{equation*}
 which we define   
  as  a maximization  as the feasible region is~compact.

 {Note that whenever} 
 $\calK_1\subseteq \calK_2$ we have that   $\omega(\calK_1,\calG)\le \omega(\calK_2,\calG)$.  By Proposition \ref{cor:dnnupperbound} it follows  that  $ \omega(\DNN, \calG)$ is an SDP      upper bound to the quantum value of  a nonlocal game, that never exceeds the no-signaling value. 
 
\begin{proposition}\label{thm:sdpupperbound}
For any  game $\calG$  we have   
$\qvalue\le \omega(\DNN,\calG)\le \nsvalue.$
\end{proposition} 
   
As it turns out,  this SDP  
  was already studied by Feige and Lov\'asz as an upper bound to the {\em classical value}  of an arbitrary  nonlocal game (cf. Equations (5)-(9) in \cite{FL92}).  On the other hand,   Proposition~\ref{thm:sdpupperbound} yields   a much stronger result, namely that   $\omega(\DNN,\calG)$   is in fact an upper bound to the quantum value, so in particular it also upper bounds $\cvalue$.  To the best of our knowledge, prior to this work,  the only known  result relating $\omega(\DNN,\calG)$ 
    with the quantum value is that they are equal for  XOR games~\cite[Theorem 22]{S11}.  
  
We conclude this section  by comparing  $\omega(\DNN,\calG) $ with the maximum probability of winning the game $\calG$  when the players use strategies that generate correlations in  the first level of the NPA hierarchy, denoted by ${\rm SDP}^{(1)}(\calG)$. 
As an immediate consequence of Theorem \ref{thm:corrnpa} we have that:

\begin{proposition}\label{sdpupperbound}
For any game $\calG$ we have that $\omega(\DNN,\calG)\le {\rm SDP}^{(1)}(\calG).$
\end{proposition}
\noindent At present, we have not been able to identify a    game  for which this inequality is strict.
{Lastly, by Remark \ref{NPAalternative} it is easy to see  that  ${\rm SDP}^{(1)}$ is equal to:}  
\be\label{sdpkrt}
\begin{aligned}
\text{ maximize} & \quad    \la \hat{C}, X\ra  &  \\ 
   \text{ subject to} & \quad    \la J_{i,j}, X\ra  =1,  \ \text{ for all } i, j\in S\cup T,  \\ 
    & \quad X[(s,a),(t,b)] \ge 0, \  \text{ for  all }\  s\in S,t\in T, a\in A, b\in B,  \\
  &  \quad X[(s,a),(s,a')] =0, \ \text{ for all } s\in S, a\ne a'\in A,\\ 
   &    \quad X[(t,b),(t,b')] =0,\  \text{ for all } t\in T, b\ne b'\in B,\\ 
 &\quad  X \in  \pos^{\asbt}.
 \end{aligned}
 \ee
  
 The SDP  given in  \eqref{sdpkrt} is the  ``canonical'' SDP  relaxation for $\qvalue$ that is usually considered in the quantum information  literature {(e.g. see~\cite{KRT})}.
 
\paragraph{{\bf Perfect strategies.}} 
\label{sect:perfect}

 Recall that    for a convex cone  $\calK \subseteq \NN$ we say  {that}  the game  $\calG(\pi,V)$ admits  a {\em perfect} $\calK$-strategy if 
$\kvalue=1$ and moreover, this value is achieved by some correlation in $\corr(\calK)$ (cf. Definition \ref{perfectkstrategy}). 
Using our conic formulations we show  that deciding the existence of a perfect $\calK$-strategy for an arbitrary nonlocal game  can be cast as a conic program over the cone~$\calK$. 
 
\begin{lemma}\label{thm:perfect_strategy}
 Let $\calG(\pi, V)$ be a  game with question sets $S, T$ and answer sets $A, B$  and let  $\calK \subseteq~\NN^N$. The game  $\calG$ admits  a perfect $\calK$-strategy  if and only if  the  following conic program is feasible: 
{\be\label{perfectstrategies_conicformulation}\tag{$\mathcal{F}_\calK$}
\begin{aligned}
 & X \in \calK, \quad \inner{J_{i,j}}{X} = 1,\  \forall  i,j \in S \cup T,  \\
 & \xasbt = 0, \forall a,b,s,t \text{ with } \pi(s,t) > 0 \text{ and } V({a,b|s,t} )= 0. \\ 
 \end{aligned} 
\ee}
\end{lemma}

\begin{proof}
For any $\pabst\in \corr(\calK)$ we have that 
\be\label{scefergvw}
{\sum_{s,t} \pi(s,t) \sum_{a, b} V(a,b|s,t) p(a,b|s,t) 
\leq 
\sum_{s,t} \pi(s,t) \sum_{a, b}p(a,b|s,t)
= 1.}
\ee 
{Therefore $\calG$ admits a perfect $\calK$-strategy if and {only if} \eqref{scefergvw} holds throughout with equality  for some   $\pabst\in\corr(\calK)$.}
This is equivalent to  
\[ \pi(s,t) (V(a,b|s,t) - 1) p(a,b|s,t) = 0, \text{ for all } a\in A,b\in B,s\in S,t\in T,\]
which {shows that} 
  $p(a,b|s,t)=0$ when $  \pi(s,t)>0 \text{ and } V(a,b|s,t)=~0.$
\end{proof} 
 
Lemma \ref{thm:perfect_strategy}, combined with  Theorem \ref{thm:correlations} {implies} the following.

\begin{corollary}\label{conicformulation_perfectstrategies}For any nonlocal game $\calG$ we have that:

\bi
\item[(i)] $\calG$ admits a perfect classical strategy if and only if $(\mathcal{F}_{\CP})$ is feasible. 
\item[(ii)] $\calG$ admits a perfect quantum  strategy if and only if $(\mathcal{F}_{\CPSD})$ is feasible. 
\item[(iii)] $\calG$ admits a perfect no-signaling strategy if and only if $(\mathcal{F}_{\nso})$ is feasible. 
\item[(iv)] $\calG$ admits a perfect unrestricted strategy if and only if $(\mathcal{F}_{\NN})$ is feasible. 
\ei
\end{corollary} 
 
We conclude  this section  with an  equivalent form of {Corollary}~\ref{conicformulation_perfectstrategies}  which  is used in Section \ref{sec:perfectcorrelated}. {This follows easily using Lemma~\ref{lem:dnnextrarow}.}

\begin{proposition}
\label{thm:conicformulation_perfectstrategies1} 
{For any  $\calK\in \{\CP,\CPSD,\DNN\}$,  a nonlocal game $\calG(\pi,V)$ admits a perfect $\calK$-strategy if and only if  there exists  $\tilde{X}\in \calK^{1+N}$ satisfying:} 
\bi
\item[$\bullet$]$\la J_{i,j},\tilde{X}\ra=1, \text{ for all } i,j\in\{0\}\cup S\cup T, \text{ and }$
\item[$\bullet$] $ \tilde{X}[(s,a),(t,b)]  = 0,  \ \forall  a,b,s,t \text{ with } \pi(s,t) > 0 \text{ and } V({a,b|s,t} )=~0.$
\ei
\end{proposition}
 
\section{Synchronous correlations and game values}
\label{sect:corr} 
 
Throughout this  section we only consider  Bell scenarios   where $A=B$ and $S=T$. Given such a scenario  we study  {\em synchronous} correlations, i.e., correlations  with the property that the players respond with   the same answer whenever they receive  the same question. First, we specialize    Theorem~\ref{thm:correlations}  to    synchronous correlations and show   that our conic characterizations  assume a particularly simple form. Based on these simplified characterizations   we study the maximum probability of winning a nonlocal game  
using strategies that generate  synchronous correlations. This allows us to derive  conic programming formulations for deciding  the existence of perfect strategies  for synchronous  nonlocal games. As a corollary,   we  recover  in a uniform manner  the conic programming formulations for deciding the existence of a classical and quantum  graph homomorphisms \cite{R14b} and also  the conic programming formulations for the quantum chromatic and the quantum independence number~\cite{LP14}.

\paragraph{{\bf Synchronous correlations.}}
\label{sec:syncorrelations} We start this section with a central definition.

\begin{definition}\label{synchronous_correlation}
A correlation $p\in \calP$ is called {\em synchronous} if  the players  always respond with   the same answer upon receiving  the same question, i.e., 
\be\label{synchronous}
p(a,a'|s,s)=0, \  \text{ for all } s\in S \text { and } \anea.
\ee
\end{definition} 
  
 Our first result  is a geometric lemma  that is essential in obtaining  simplified conic characterizations for quantum and classical synchronous correlations.  

\begin{lemma}\label{thm:psdvectorlemma} Let $\mcx$ be a Euclidean
space  and  consider two families of  psd matrices  
$\{ X_i : i \in [n] \}\subseteq  \mch_+(\mcx)$ and $\{ Y_i : i \in [n] \}\subseteq \mch_+(\mcx)$ satisfying: 
\begin{itemize} 
\item[(i)] $\sum_{i=1}^n X_i = \sum_{i=1}^n Y_i$, \ and 
\item[(ii)] $\inner{X_i}{Y_{j}} = 0, $\  for all $i\ne j \in [n]$. 
\end{itemize} 
Then we have that $X_i = Y_i\ $ for all $i \in [n]$. 
\end{lemma} 

\begin{proof}
{Fix 
$i\in [n]$} and let $\lambda$ be the largest eigenvalue of $X_i$ with corresponding (normalized) eigenvector $v$ and let $\mu$ be the largest eigenvalue of $Y_i$. By  condition $\textit{(ii)}$, we know that $Y_{j} v = 0$ for all $j \neq i\in [n]$. Using this, we have 
\be\label{psdvectorlemma} 
\lambda 
= 
v^* X_i v 
\leq 
v^* ( \sum_{j=1}^n X_j ) v 
= 
v^* ( \sum_{j=1}^n Y_j ) v 
= v^* Y_i v
\leq 
\mu, 
\ee
proving that $\mu \geq \lambda$. By symmetry, we also have $\lambda \geq \mu$ proving that $\lambda=\mu$. This shows that  \eqref{psdvectorlemma} holds throughout with equality and thus $v$ is also an eigenvector of $Y_i$  corresponding to eigenvalue $\lambda = \mu$  as well.
Lastly, define   $X_i':=X_i-\lambda vv^*, Y_i':=Y_i-\lambda vv^*$ and for $j\ne i$ set $X'_j:=X_j$ and $Y_j':=Y_j$. Notice that the matrices {$\{ X_i'\}_{i=1}^n\subseteq \mch_+(\mcx)$} and {$\{ Y_i'\}_{i=1}^n\subseteq \mch_+(\mcx)$} satisfy conditions  \textit{(i)} and \textit{(ii)} and  the proof is concluded by an inductive argument. 
\end{proof} 

Based on   Lemma \ref{thm:psdvectorlemma} we now  derive  a second result    that we use in Theorem~\ref{thm:synchronous} below and in our study of  perfect strategies. 
   
\begin{lemma}\label{lemma:main}
Consider a family of vectors $ \sxsa $    in some 
Euclidean space~$\mcx$.
\bi 
\item[(a)] For   $X := \gram(\sxsa)$
the following are equivalent:  
\bi
\item[(i)] There exists $k\in \mcx$ satisfying  $\sum_{a\in A} x^s_a=k$ for all $s\in S$   and $\la k,k\ra=1$. 
\item[(ii)] $ \la J_{s,s'}, X\ra=1,\ $ for all $s,s'\in S$. 
\ei 
\item[(b)] Set $\tX:=\gram(k,\sxsa)$ where  $k\in \mcx$ with $\la k,k\ra=1$.  The following are equivalent:
\bi
\item[(i)] $\sum_{a\in A} x^s_a=k,\ $  for all $s\in S$.
\item[(ii)]  $ \la J_{s,s}, \tX\ra=\la J_{0,s}, \tX\ra=1,\ $ for all $s\in \{0\}\cup S$.
\ei
\ei
\end{lemma}
\begin{proof}{The proof is similar to the proof of Lemma \ref{lem:vectors} and we omit most cases}. 
We only consider case \textit{(b)} and show that  \textit{(ii)} implies \textit{(i)}. Notice that 
\be\label{whatever}
\big\la k-\sum_{a\in A} x^s_a, k-\sum_{a\in A} x^s_a\big\ra =\la k,k\ra-2\sum_{a\in A}\big\la k,x^s_a\big\ra+\big\la \sum_{a\in A}x^s_a, \sum_{a\in A}x^s_a\big\ra.
\ee
By assumption we have that $\la k,k\ra=1, $  $\la J_{0,s}, \tX\ra=\sum_{a}\big\la k,x^s_a\big\ra=1 $
and $\big\la \sum_{a}x^s_a, \sum_{a}x^s_a\big\ra=\la J_{s,s},\tilde{X}\ra=1$. Substituting   in \eqref{whatever} the proof is concluded.
\end{proof}
  
We now  arrive at  the  main result in this section.
 
\begin{theorem}
\label{thm:synchronous}
Consider a  Bell scenario with  question set $S$ and answer set $A$. Furthermore, consider a     {\em synchronous} correlation  $\paass\in \calP$  and  set $P:=\sum_{a,a',s,s'} p(a,a'|s,s')\ e_se_{s'}^\sfT\otimes e_ae_{a'}^\sfT.$  For      every cone  $\calK\in \{\CP,\CPSD\}$ the following are equivalent:
\bi
\item[(i)]  $p\in\corr(\calK)$.
\item[(ii)] $P\in \calK^{|S\times A|}.$
\item[(iii)]  There exists  
$\tilde{X}=\begin{pmatrix}
1 & x^\sfT \\
x & P
\end{pmatrix}\in \calK^{1+|S\times A|}$ { s.t. } 
 $ \la J_{0,s}, \tilde{X}\ra=1, \ \forall s\in \{0\}\cup S$.
\ei
\end{theorem}
 
\begin{proof} We only consider $\calK=\CPSD$ the case  $\calK=\CP$ being  similar. 
 
 $\textit{(i)}\Longrightarrow \textit{(ii)}$:
  Since  $\paass\in \corr(\CPSD)$,  by  Theorem \ref{thm:qcorrelationvector}  there exist psd matrices  $\xsa,$ $\{Y^s_a\}_{s,a},$ $  K\in \pos^d$ (for some $d\ge 1$)  such that $p(a,a'|s,s')=\la X^s_a, Y^{s'}_{a'}\ra$ for all $a,a',s,s', $   $\sum_{a}X^s_a=\sum_{a} Y^s_a=K$ for all $s\in~S$ and $\la K,K\ra=1$. Since  $p$ is synchronous it follows from  Lemma \ref{thm:psdvectorlemma} that $X^s_a=~Y^s_a$ for all $s\in S$ and $a\in A$. This implies that $P\in \CPSD^{|S\times A|}$.
  
   $\textit{(ii)}\Longrightarrow \textit{(iii)}$: Since $P\in \CPSD^{|S\times A|}$  there exist matrices $\xsa \in \pos^d$ (for some $d\ge 1)$  such that $p(a,a'|s,s')=\la X^s_a, X^{s'}_{a'}\ra$ for all $a,a',s,s'$. By Lemma~\ref{lemma:main} \textit{(a)} there exists $K\in \pos^d$ such that $\la K,K\ra=1$ and  $\sum_{a}X^s_a=K$ for all $s\in S$. The proof is concluded by {noticing that  
   $\gram(K,\xsa)$ is feasible for  \textit{(iii)}}. 
   
    $\textit{(iii)}\Longrightarrow \textit{(i)}$: Let  $\tilde{X}\in \CPSD^{1+|S\times A|}$ be feasible for \textit{(iii)} and  
 {consider $K,$ $\xsa\in~\pos^d$ such that    $\tilde{X}=\gram(K,\xsa)$.}  {Since $p\in \calP$ we have that $\la J_{s,s},\tilde{X}\ra=1$, for all $s\in S$.} Thus, by Lemma \ref{lemma:main}~\textit{(b)} we have that
    $\sum_{a} X^s_a=K,$ for all $s\in S$.  Lastly, Theorem~\ref{thm:qcorrelationvector}  {implies  ${p\in \corr(\CPSD)}$}.
\end{proof}
 
{As a consequence of Theorem \ref{thm:synchronous}  we arrive at  the following conic characterization of the sets of synchronous quantum and  classical correlations:}

\begin{corollary}
\label{cor:synchrstrategies}
 Consider a  Bell scenario with  question set $S$ and answer set $A$ and let  $\calK\in \{\CP,\CPSD\}$. The set of synchronous $\calK$-correlations  is given by
$$\{ X\in \calK^{|S \times A|}: \la J_{s,s'},X\ra=1, \text{ for } s,s' \text{ and } X[(s,a),(s,a')]=0,  \text{ for }a\ne a', s\},$$
where we identify a correlation vector $\paass$  with {the square matrix 
\[ P := \sum_{a, a', s, s'} p(a, a'|s, s') \, e_s e_{s'}^{\sfT} \otimes e_a e_{a'}^{\sfT}. \] 
}
\end{corollary} 
 
\paragraph{{\bf Synchronous  value.}}
\label{sec:synchvalue}  
{In this section we  study the value of a nonlocal game when the players use strategies that generate  synchronous correlations.} 
\begin{definition}\label{def:ksyncvalue}For any  convex cone $\calK\subseteq \NN^N$,  the {\em $\calK$-synchronous value} of a nonlocal game $\calG$, denoted $\omega_{syn}(\calK,\calG)$,  is defined as  the maximum probability   of winning the game  when the players are 
only allowed to use strategies that generate synchronous  $\calK$-correlations. 
\end{definition}

As we now show by  Corollary \ref{cor:synchrstrategies} we get   a conic programming formulation  for the classical and quantum synchronous value of a nonlocal game with matrix variables of size $|S\times A|$.

\begin{proposition} \label{thm:smallCPs}
Consider a game $\calG$ with question  set $S$ and answer set $A$. For a cone  $\calK\subseteq \DNN^{|S\times A|}$ define   
\begin{align*}
\nu(\calK, \calG) := 	\textup{supremum}\quad & {1\over 2}\inner{C+C^\sfT}{X} \\
  \textup{subject to} \quad  & \inner{J_{s,s'}}{X} = 1, \text{ for all } s,s' \in S, \\
& X[{(s,a),(s,a')}] = 0, \text{ for } s\in S,  \anea, \\ 
& X \in \calK^{|S\times A|}.
\end{align*}  
Then $\omega_{syn}(\CP,\calG)=\nu(\CP, \calG)$ and  
$\omega_{syn}(\CPSD, \calG)=\nu(\CPSD, \calG)$.  
\end{proposition} 

{Note that in the proposition above we use ${1\over 2}\inner{C+C^\sfT}{X}$ as the objective function. This is 
because this is equal to  $\sum_{a,a',s,s'} \pi(s,t) V(a,a'|s,s') \, p(a,a'|s,s')$, which is  exactly the probability Alice and Bob win the game.}  

There are many examples of  games for which the optimal classical strategy generates a synchronous correlation  but
{optimal quantum correlations are not synchronous} 
 (e.g. the CHSH game). This raises the following question: What is the optimal value for such games when one restricts to synchronous  quantum strategies? Perhaps the interesting thing to see is if the power of quantum strategies comes from the fact that the optimal quantum strategies do not need to be synchronous.
The following corollary of Proposition~\ref{thm:smallCPs} gives a partial answer to the above question. It is a consequence  of  the fact that $\CP^n= \CPSD^n = {\DNN}^n$ for any $n \leq 4$~\cite{LP14}.  
 
\begin{corollary} \label{cor:smallgames}
For any nonlocal game $\calG$ with  identical binary question sets (i.e., $S=T$ and $|S| = 2$) and identical binary answer sets (i.e., $A=B$ and $|A| = 2$), we have that the synchronous classical and synchronous  quantum values coincide  and are  expressible as a semidefinite program, i.e., 
\be\label{whateverrrrrrrr}
\omega_{syn}(\CP, \calG) = \omega_{syn}(\CPSD, \calG) = \nu({\DNN}, \calG). 
\ee
\end{corollary} 

{As an example, consider the CHSH game for which  there exists an optimal classical strategy which is synchronous (Alice and Bob just output $0$) with success probability $3/4$. Then, it follows from Corollary~\ref{cor:smallgames} that the  synchronous  quantum value is also $3/4$. That is, quantum strategies cannot be synchronous to win CHSH with greater probability than classical strategies.} 
 {For games with large question or answer sets, Corollary~\ref{cor:smallgames} does not help. However, Proposition~\ref{thm:smallCPs} implies that $\nu({\DNN},\calG)$ is a tractable upper bound on the synchronous quantum value of~$\calG$.}
  
\paragraph{{\bf Perfect strategies for synchronous  games.}} 
\label{sec:perfectcorrelated}  
{In this section we focus on a class of nonlocal games for which any 
perfect strategy {generates} 
 a synchronous correlation}.
\begin{definition}\label{def:correlatedgame}
A nonlocal game $\calG=(\pi,V)$ is called {\em synchronous} if both  players share the same question set $S$ and  the same answer set $A$,  and
$$V(a,a'|s,s)= 0, \; \text{ for all } s \in S, \anea, \text{ and } \pi(s,s) >0,  \  \text{ for all } s \in S.$$
\end{definition}  
    
{By {Corollary}~\ref{conicformulation_perfectstrategies}},  deciding the existence of a perfect $\calK$-strategy is equivalent to the feasibility of a linear conic program with matrix variables of size {${|(S\times A) \times (T\times B)|}$}.  Moreover,  by  
Lemma \ref{thm:perfect_strategy}, any  perfect strategy for a synchronous game generates a synchronous correlation. 
{Thus we can}  
use  Theorem~\ref{thm:synchronous} to derive a conic program with matrix variables of size $|S\times A|$ whose feasibility is equivalent to the existence of a perfect  $\calK$-strategy.  

\begin{theorem}
\label{thm:conicformperfectstrategies}
Let $\calG(\pi, V)$  be a synchronous game and  $\calK\in \{\CP,\CPSD\}$. The  following are equivalent:
\bi
\item[(i)]$\calG$ admits a perfect $\calK$-strategy.
\item[(ii)]  There exists a matrix  $X\in \calK^{|S\times A|}$ satisfying:
  \begin{itemize}
   \item[$\bullet$] $\la J_{s,s'},X\ra=1, \, \forall s,s', $
   \item[$\bullet$]  $X[(s,a),(s,a')]    = 0, \, \forall s, a \neq a',$
     \item[$\bullet$]  {$X[(s,a),(s',a')]  = 0, 
\, \forall a,a',s,s'  \text{ with }  \pi(s,s') > 0 \text{ and } V({a,a'|s,s'} )=~0.$}
\end{itemize}
\item[(iii)] There exists a matrix $\tilde{X}\in \calK^{1+|S\times A|}$ satisfying:
\begin{itemize}
\item[$\bullet$] $\la J_{s,s},\tilde{X}\ra=\la J_{0,s},\tilde{X}\ra=1, \, \forall s\in \{0\}\cup S,$
   \item[$\bullet$]  $ \tilde{X}[(s,a),(s,a')]    = 0, \, \forall s, a \neq a', $
    \item[$\bullet$] {$X[(s,a),(s',a')]  = 0, 
\, \forall a,a',s,s'  \text{ with }  \pi(s,s') > 0 \text{ and } V({a,a'|s,s'} )=~0.$} 
\end{itemize}
\ei
\end{theorem}

The proof is omitted as it is an easy consequence of   Theorem \ref{thm:synchronous}. In the next two sections we specialize  Theorem \ref{thm:conicformperfectstrategies} to graph coloring and more generally, graph homomorphism games and derive  conic formulations for the existence of perfect strategies for  these  classes of  games. 
 
\paragraph{{\bf Graph Homomorphisms.}}
\label{sec:homomrph} 
 
Given two undirected graphs $H$ and $G$, 
a {\em graph homomorphism} from $H$ to $G$, denoted $ H\rightarrow  G$,  is  an adjacency preserving map from the vertex set of $H$ to the vertex set of $G$, i.e., a function  ${f: V(H)\rightarrow  V(G)}$ with the property that $f(h)\sim_G f(h')$ whenever $h\sim_H h'$.
Here we study the   {\em $(H,G)$-homomorphism  game} where   Alice and Bob  are trying to convince a referee that there exists a graph homomorphism from  $H$ to  $G$.  To verify their claim the   referee  sends  each player  a vertex of  $H$. Each player responds  with a vertex of  $G$. The player's answers   model a homomorphism $f: H\rightarrow G$, i.e., the answer to a question $h \in V(H)$ should be~$f(h) \in V(G)$.  
 
Formally, in the $(H,G)$-homomorphism  game the players share the same question set ${S:=V(H)}$ and the same answer set $A := V(G)$. 
The distribution  over the question set  is the uniform distribution on $\{(h,h): h\in V(H)\}\cup \{(h,h'): h\sim_H h'\}$. Lastly, the verification predicate is given~by 
$$V(g,g'|h,h')=\begin{cases} \; 0, \quad  \text{ if } h=h' \text{ and }g\ne g,'\\
\; 0, \quad \text{ if } h\sim_H h'\text{ and } (g\not\sim_G g \text{ or } g=g'),\\
\; 1,  \quad \text{ otherwise}.
\end{cases}$$

Notice  that the existence of a  homomorphism  $H\rightarrow G$ can be understood via the  $(H,G)$-homomorphism game. Specifically, there exists a graph homomorphism from $H$ to $G$ if and only if  the  $(H,G)$-homomorphism game admits a perfect  classical strategy. This is easy to see  using the fact  that in the classical setting  we may assume without loss of generality  that both players are using deterministic strategies.  This motivates the following definition.

\begin{definition} 
For  graphs $G$ and $H$, we say there exists a {\em quantum graph homomorphism} from  $H$ to  $G$, denoted $H\overset{q}{\rightarrow} G$,  if the $(H,G)$-homomorphism game admits a perfect quantum strategy. 
\end{definition} 

Quantum graph homomorphisms were  introduced  recently in~\cite{MR14b}. 
Using our conic characterizations for  the sets of quantum and classical correlations we arrive at  a natural conic generalization of the notion  of graph~homomorphism.

\begin{definition}For a convex cone $\calK\subseteq \NN$ we say that there exists a {\em $\calK$-homomorphism} from $H$ to $G$ if and only if the $(H,G)$-homomorphism game admits a perfect $\calK$-strategy. 
\end{definition} 

By Lemma \ref{thm:perfect_strategy},  the existence of a $\calK$-homomorphism from $H$ to $G$ is equivalent to the feasibility of a linear conic program over  $\calK$. The similar notion of {\em strong $\calK$-homomorphism}   was  introduced recently  in \cite{R14b}. Since the $(H,G)$-homomorphism game is synchronous  we can use the conic formulations from  Theorem \ref{thm:conicformperfectstrategies} to show  that  the two notions of conic homomorphisms coincide  for  $\calK\in \{\CP,\CPSD\}$. We note that    strong $\calK$-homomorphisms are only defined for a certain class of convex cones called {\em frabjous}. Working over frabjous cones  ensures that the   $\calK$-homomorphism relation is reflexive and transitive, mimicking  classical graph homomorphisms.
  
As an immediate consequence of   Theorem \ref{thm:conicformperfectstrategies} \textit{(ii)} it follows  that deciding the existence of a classical (resp. quantum) graph homomorphism can be formulated  as a feasibility  conic program over the cone of completely positive (resp. completely positive semidefinite) matrices. 
\begin{corollary}\label{thm:conichomomor}
Consider  two 
graphs  $H$ and $G$ and let  $\calK\in \{\CP,\CPSD\}$. The $(H,G)$-homomorphism game admits a perfect $\calK$-strategy if and only if there exists $X\in \calK^{|V(H)\times V(G)|}$  such that 
{\bi
  \item[$\bullet$] $\sum_{g \in V(G)} \sum_{g'\in V(G)} X[(h,g),(h',g')] =1, \text{ for all } h,h'\in  V(H), \text{ and }$ 
 \item[$\bullet$]  $X[(h,g),(h',g')] = 0,  \text{ when } 
(h=h' \text{ and } g\ne g') \text{ or } (h\sim_H h' \text{ and } g\not \sim_G~g').$
\ei} 
\end{corollary} 

{We note that the case $\calK=\CP$ (resp. $\calK=\CPSD$) corresponds to  Theorem~4.1 in \cite{R14b}  (resp.  Theorem 4.3 in \cite{R14b}).}
   
\paragraph{{\bf Chromatic and independence number.}}
\label{sec:chromatic} 
 
A {\em $k$-coloring} for a graph $G$ corresponds to an assignment of  one out of $k$ possible colors to its vertices  so that adjacent vertices receive different colors.    
The {\em chromatic number} of a graph $G$, denoted $\chi(G),$ is equal to the smallest integer $k\ge1$ for which $G$ admits a $k$-coloring.  
Notice that $G$ admits a $k$-coloring if and only if there exists  a  homomorphism from $G$  into $K_k$, i.e., the complete graph on $k$ vertices. Thus,  $\chi(G)$ may be equivalently defined as the smallest  $k\ge 1$ for which the $(G,K_k)$-homomorphism game  admits a perfect classical strategy. {The {\em quantum chromatic number} of a  graph $G$, denoted $\chi_q(G)$,  is equal to the smallest  $k\ge 1$ for which  the $(G,K_k)$-homomorphism game admits a perfect \emph{quantum} strategy.}

The {\em independence number} of a   graph $G$, denoted $\alpha(G)$,  is equal to the largest number of pairwise nonadjacent vertices of $G$. Notice that $G$ contains $k$ pairwise nonadjacent vertices if and only if there exists a homomorphism from $K_k$ into  $\overline{G}$, where $\overline{G}$ denotes the complement of the graph $G$. As a result,   $\alpha(G)$ can be equivalently defined as the largest  integer $k\ge 1$ for which the $(K_k, \overline{G})$-homomorphism game admits a perfect classical strategy. {Analogously, the {\em quantum  independence number} of  $G$, denoted $\alpha_q(G)$,  is the largest   $k\ge 1$ for which the $(K_k, \overline{G})$-homomorphism game admits a perfect \emph{quantum} strategy.}

The quantum chromatic number was introduced and studied in \cite{Cameron07} and the quantum independence number in \cite{ML15}.  It was recently shown that deciding whether the quantum chromatic number of a graph is at most~$3$ is NP-hard~\cite{Ji13}.
 
Using Theorem \ref{thm:conicformperfectstrategies}  \textit{(iii)}  we immediately get  conic programming formulations  for  the quantum chromatic number and the quantum independence number of a graph. We note that these formulations  were also identified  in Proposition 4.10  and  Proposition 4.1 in \cite{LP14}, respectively. 
 
\begin{corollary} The quantum chromatic number of a 
graph $G$ is equal to the smallest integer $k\ge 1$ for which {there exists 
$X    \in \CPSD^{|V(G)| k+1}$ satisfying:}  
{\bi 
 \item[$\bullet$] $X[0,0] =1$;
  \item[$\bullet$]  $\sum_{i,i'\in [k]} X[(g,i),(g,i')] =\sum_{i\in [k]} X[0,(g,i)]=1,\text{ for all }   g\in V(G)$;
  \item[$\bullet$] $X[(g,i),(g',i')]= 0,  \text{ when } 
(g=g' \text{ and } i\ne i' ) \text{ or } (g\sim g' \text{ and } i=i')$.
\ei}
 
\noindent The quantum independence number of a  
graph $G$ is equal to the largest integer $k\ge 1$ for which there exists a matrix    $X    \in \CPSD^{k |V(G)|+1}$ satisfying:
{\bi 
  \item[$\bullet$] $X[0,0] =1$;
 \item[$\bullet$]   $\sum_{g,g'\in V(G)} X[(i,g),(i,g')] =\sum_{g\in V(G)} X[0,(i,g)]=1, \ \text{ for all } i\in [k]$;
    \item[$\bullet$] $X[(i,g),(i',g')]  = 0, 
  \text{ when }  (i=i' \text{ and } g\ne g') \text{ or }
  {(i\ne i'  \text{ and } g\simeq g')}. $
\ei} 
\end{corollary} 
 
\medskip 

\noindent{\bf Acknowledgements.} {The authors would like to thank the referees for carefully reading the paper and for their useful comments.}  Furthermore, the authors thank S. Burgdorf, M. Laurent, L. Man\v{c}inska, T. Piovesan,  D. E. Roberson, S. Upadhyay,  T. Vidick and Z. Wei  for useful discussions.  AV is supported in part by the Singapore National Research Foundation under NRF RF Award No. NRF-NRFF2013-13. JS is supported in part by NSERC Canada.  Research at the Centre for Quantum Technologies at the National University of Singapore is partially funded by the Singapore Ministry of Education and the National Research Foundation, also through the Tier 3 Grant ``Random numbers from quantum processes,'' (MOE2012-T3-1-009). 
 
\nocite{CHSH69}
\nocite{BLP15}
\nocite{BTN13} 
\nocite{BSM03}  
\nocite{Dur10}  
\nocite{FL92}
\nocite{Ji13}
\nocite{R14b} 
\nocite{S11}
\nocite{ML15}  
\nocite{NC00}
\nocite{MRV14} 
\nocite{MM61}
\nocite{CSUU08}
\nocite{DSV14}
   
\bibliographystyle{abbrv}
\bibliography{biblio}
 
\appendix 

\section{Constraint satisfaction problems} 
    
It is natural to ask what kinds of combinatorial problems admit linear conic formulations of a similar form. In  this appendix  we show that all  examples considered  in this work  can be cast in the common framework of binary  constraint satisfaction problems. 
  
An instance of a {\em constraint satisfaction problem} (CSP) is specified  by a triple   $(\mcv, \mathcal{D}, \mathcal{C} )$ where the elements of     $\mcv=\{x_1,..,x_n\}$ are called the  {\em variables} of the CSP,  the elements of $\mathcal{D}=\{D_1,\ldots,D_n\}$ are the {\em domains} of the corresponding variables and the elements of $\calC=\{C_1,\ldots, C_m\}$ are called the  {\em constraints} of the CSP. Each constraint $C_i$ involves a subset of variables $\{ x_{i_1},\ldots,x_{i_{t_i}}\}\subseteq \mcv$  and is defined as some   $t_i$-ary relation on $D_{i_1}\times \cdots \times D_{i_{t_i}}$. The number of variables $t_i$  is called the {\em arity} of the constraint $C_i$.  We say that a  CSP  is {\em satisfiable} if there exists an  assignment of  values to each  variable from its corresponding domain so that every  constraint is satisfied.  A CSP that only involves  constraints of arity $2$ is called a {\em binary} CSP.
 
Deciding the  existence of  a  homomorphism from a graph $H$ to a graph $G$ can be formulated  as an  instance of a binary CSP. Specifically,  we have  one variable for each vertex of $H$ and the  domain of each variable is the vertex set  of $G$. Lastly,  for every edge $e=(h,h')\in E(H)$ we have a constraint $C_e$ of arity  2 involving the variables $h$ and $h'$; the constraint  is given by $C_e=\{E(G)\}$. 
  
To any binary  constraint satisfaction problem   $\calP := (\mcv, \mathcal{D}, \mathcal{C} )$ we may associate a two-player nonlocal game, denoted $\calG(\calP)$,  having the  property that the CSP  is satisfiable if and only if 
  the game admits a perfect classical strategy. The game is defined  as follows: The referee selects  uniformly at random a pair of variables $(x_i,x_j)\in \mcv\times \mcv$ and sends $x_i$ to Alice and $x_j$ to Bob. For the players to win they need to respond to the referee with an element of $D_i$ and $D_j$, respectively. Furthermore, if there exists some constraint $C_k$ that involves the variables $x_i$ and $x_j$ then the  answers of the players must  satisfy the constraint. Lastly, if the players receive the same variables as questions  they have to provide identical answers ensuring that  the game is synchronous.
    
\begin{definition}
{A binary constraint satisfaction problem  $\calP$  is called}  {\em quantumly satisfiable} if the  nonlocal game $\calG(\calP)$   admits a perfect quantum strategy.
\end{definition} 
   
Notice that the notion of quantum satisfiability of binary CSP's generalizes  the concept  of quantum graph homomorphisms. Indeed, it is immediate from the definitions that there exists a quantum graph homomorphism from  $H$ to $G$ if and only if the nonlocal game corresponding to the homomorphism  CSP  admits a perfect quantum strategy. 
  
The majority of the literature concerning CSP's usually focuses on   binary CSP's. The reason for this is that   any non-binary CSP  $\calP$ can be converted to  a  binary CSP $\calP'$ such   that $\calP$ is satisfiable if and only if $\calP'$ is satisfiable. The transformation is straightforward: For each constraint $C_i$ of $\calP$ we introduce  one variable $c_i$ in  $\calP'$.  The domain of the variable $c_i$ is  given by all  assignments that satisfy   the constraint $C_i$ in $\calP$. Lastly, for every two constraints $C_i, C_j$ of $\calP$ that share a variable $x_k$  we add a binary constraint between the variables $c_i,c_j$  in $\calP'$. This constraint excludes those  satisfying assignments for $C_i$ and $C_j$  where the common variable $x_k$ receives   different values.  
  
The discussion above allows us to generalize the notion of quantum satisfiability from binary CSP's  to arbitrary ones.  Combining this  fact with  Theorem~\ref{thm:conicformperfectstrategies} we get the following corollary. 
 
\begin{corollary}Deciding whether an arbitrary constraint satisfaction problem is satisfiable (resp. quantumly satisfiable) is equivalent to deciding the feasibility of a linear conic program over $\CP$ (resp. $\CPSD$).  
\end{corollary} 
 
\end{document}